\documentclass{siamart171218}
\usepackage{mathrsfs}
\usepackage{amsmath,amssymb,amsfonts,bm,comment}
\usepackage{color}
\usepackage{enumerate}
\usepackage{hyperref}
\usepackage[utf8]{inputenc}
\usepackage{todonotes}
\usepackage{cite}
\usepackage{mathtools}
\usepackage{listings}
\usepackage{tikz}
\usetikzlibrary{arrows.meta}

\usepackage{pgfplots}
\usepackage{geometry}
\usepackage{centernot}
\usepackage{graphicx}
\newcommand{\qtoep}{\hbox{\rm qtoep}}
\newcommand{\cqttoolbox}{\texttt{cqt-toolbox}}
\newcommand{\wh}{\widehat}
\newcommand{\cu}{\mathbf i}

\newcommand{\QT}{\ensuremath{\mathcal{QT}}}
\newcommand{\QTD}{\ensuremath{\mathcal{QT}_\infty^d}}
\newcommand{\EQT}{\ensuremath{\mathcal{EQT}}}
\newsiamremark{remark}{Remark}
\newsiamremark{example}{Example}

\renewcommand{\tilde}{\widetilde}
\renewcommand{\hat}{\widehat}

\definecolor{mygray}{RGB}{230,230,230}
\definecolor{mylilas}{RGB}{170,55,241}
\newcommand\TheTitle{A computational framework for two-dimensional  random walks with restarts}
\newcommand\TheShortTitle{Computational framework for two-dimensional random walks}
\newcommand\TheAuthors{Dario A. Bini, Stefano Massei, Beatrice Meini, and Leonardo Robol}

\headers{\TheShortTitle}{\TheAuthors}

\title{\TheTitle}

\newcommand{\isti}{Dipartimento di Matematica, Università di Pisa, 
  Italy, partially supported by GNCS of INdAM}
\newcommand{\mathdep}{Dipartimento di Matematica, Università di Pisa, 
   Italy, partially supported by GNCS of INdAM}
\newcommand{\epfl}{EPF Lausanne, Switzerland. The work of Stefano Massei has been supported by the SNSF research project \emph{Fast algorithms from low-rank updates}, grant number: 200020\_178806}

\newcommand{\Endo}[1]{\mathcal{B}(#1)}

\newcommand{\K}{\mathcal K}

  \author{%
    Dario A. Bini\thanks{\mathdep},
    \email{dario.bini@unipi.it} \and
    Stefano Massei\thanks{\epfl},
    \email{stefano.massei@epfl.ch} \and
    Beatrice Meini\thanks{\mathdep},
    \email{beatrice.meini@unipi.it} \and
    Leonardo Robol\thanks{\isti\
      \email{leonardo.robol@unipi.it}}
}

\newcommand{\norm}[1]{\lVert #1 \rVert}

\ifpdf
\hypersetup{
  pdftitle={\TheTitle},
  pdfauthor={\TheAuthors}
}
\fi
\newcommand{\w}{{_{_{\mathcal{W}}}}}
\newcommand{\f}{{_{_\mathcal F}}}
\newcommand{\qt}{{_{_\mathcal{QT}}}}

\pgfplotsset{compat=1.13}
\begin{document}
\maketitle

\begin{abstract}
The treatment of two-dimensional random walks in the quarter plane leads to
Markov processes which involve semi-infinite matrices having Toeplitz
or block Toeplitz structure plus a low-rank correction. 	
	We propose an extension of the framework introduced in [Math. Comp., 87(314):2811–2830, 2018] which allows to deal
        with more general situations such as processes involving
        restart events. This is motivated by the need for modeling
        processes that can incur in unexpected failures like computer
        system reboots. 
We present a theoretical analysis of an enriched Banach algebra that,
combined with appropriate algorithms, enables the numerical
treatment of these problems.  The results are applied to the
solution of bidimensional Quasi-Birth-Death processes with
        infinitely many phases which model random walks in the quarter
        plane, relying on the matrix analytic approach. 
The reliability of our approach is confirmed by 
        extensive numerical experimentation on several case studies.
\end{abstract}

\section{Introduction}
\label{sec:introduction}
The treatment of the infinite data structures arising from Markov
processes usually relies on the assumption that jumps between states
become unlikely when their mutual distance increases. For instance,
this is natural when considering random walks on lattices where the
particle is forced to move to nearby positions at each step. However,
there are models that incorporate a global communication with a
certain subset of states. A rich source of case studies comes from
random walks with restart. This topic has been analyzed under
different perspectives
\cite{montero2016directed,montero2013monotonic,manrubia1999stochastic,
  evans2011diffusion}. Including resetting events is required in
various applications such as modeling computer system reboots
\cite{montero2013monotonic}, intermittent searches involved in
relocation phases of foraging animals
\cite{evans2011diffusion,benichou2005,lomholt2008levy} and computing
network indices \cite{janson2012hitting,avrachenkov2015hitting}.
Another example arises in computing return probabilities in certain
double Quasi-Birth-Death (QBD) processes: as shown in \cite[Section 5.2]{bmmr}, it can
happen that the probability of going back to a certain state, in
finite time, is positive independently of the starting position. An
analogous situation is encountered in \cite{zsh} in the case of an
M/T-SPH/1 queue system.

 In many queueing models, transition probabilities only depend on the
 mutual distances between the states. In this case it is possible to
 handle systems with infinite state space by means of a finite number
 of parameters. Moreover, this feature translates in considering
 semi-infinite matrices which have a Toeplitz structure, i.e.,
 matrices $T(a)=(t_{i,j})_{i,j\in\mathbb Z^+}$ such that
 $t_{i,j}=a_{j-i}$ for some given sequence $a=\{a_k\}_{k\in\mathbb
   Z}$, where  $\mathbb Z^+$ is the set of positive integers.
Indeed, Toeplitz matrices, finite or infinite, are
 almost ubiquitous in mathematical models where shift invariance
 properties are satisfied by some function.

Computing the invariant probability measure $\pi$ of random walks in
the quarter plane is a non trivial task due to the
infinite-dimensional nature of the model. In \cite{gos19} and \cite{gos16}, the
problem is faced by looking for representation of $\pi$ given in terms
of countably infinite sums of geometric terms. This strategy restricts
the applicability of this technique to a limited number of problems
which exclude certain transitions. On the other hand, the {\em Matrix
  Analytic Method} of M.~Neuts \cite{neuts:book} provides a more
general representation of $\pi$ given in terms of the minimal
nonnegative solution of a suitable quadratic matrix equation under no
restriction on the allowed transitions. In \cite{bmm}, \cite{bmmr}, a
framework has been introduced to handle such problems in the case where
the coefficients in the equation are matrices of infinite size,
making it possible to compute an
arbitrary number of components of $\pi$ in a finite number of
arithmetic operations.  However, this approach does not allow to deal
with models where some restart condition is involved.  In fact, in
\cite{bmm} the authors introduce the class $\mathcal{QT}$ of
semi-infinite matrices which can be approximated by the sum of a
semi-infinite Toeplitz matrix and a correction with finite support,
i.e., with a finite number of nonzero entries. But this class cannot
deal with models involving long-distance jumps, like the one occurring
in restarts, as well as in double QBDs in the cases where the
probability of going back to a certain state, in finite time, is
positive as in the example of \cite[Section 5.2]{bmmr}, or as in the
case of an M/T-SPH/1 queue system of \cite{zsh}.

In this paper, we propose a generalization of the class $\mathcal{QT}$
which includes corrections defining bounded linear operators in
$\ell^\infty$  with possibly unbounded support. The only restriction is that
the values of the entries stabilize when moving along each column. We
show that this is a suitable framework for studying models with
restarts and allows to weaken some assumptions made in \cite{bmm},
simplifying the underlying theory. Then, we present an application to
the analysis of QBD processes modeling random walks in
the quarter plane.

More specifically, we introduce the classes $\QTD$ and $\mathcal{EQT}$,
which are sets of 
semi-infinite matrices with bounded infinity norm. The former is made
by matrices representable as a sum of a Toeplitz matrix and a
correction, which represents a compact operator, with columns 
having entries which decay to zero. 
The latter is formed by matrices in $\QTD$ plus a further correction of the kind
$ev^T$ for $e^T=(1,1,\ldots)$ and $v=(v_i)_{i\in\mathbb Z^+}$ such
that $\sum_{i=1}^\infty |v_i|<\infty$. We prove that $\QTD$ and
$\mathcal{EQT}$ are Banach algebras, i.e., they are Banach spaces with
the infinity norm, closed under matrix multiplication. Moreover,
matrices in both classes can be approximated up to any precision by a
{\em finite number} of parameters. This allows to handle these classes
computationally and to apply numerical algorithms valid for finite
matrices to the case of infinite matrices.  We also show the way of
modifying the Matlab Toolbox \cqttoolbox{} of \cite{bmr} in order to
include and operate with these extended classes.  As a result, we may
effectively extend the Matrix Analytic Method of M.~Neuts
\cite{neuts:book} to the case of infinitely many states still keeping
the nice numerical features valid for the finite case. In this way we
can overcome the difficulty of the Neuts approach, pointed out by
Miyazawa in \cite[Sect. 4.3.1]{miyazawa2011light} where he writes ``it
is also well known that it (the matrix analytic method) can be used
for countably many background states, although it generally looses the
nice feature for numerical computations''.

The introduction of the new classes $\QTD$ and
$\mathcal{EQT}$ allows to handle cases which were not treatable with
the available classes, typically when restart is implicitly or
explicitly involved in the model as in the cases of \cite[Section
  5.2]{bmmr} and \cite{zsh}.
Relying on the above classes we derive some properties of the
minimal nonnegative solution $G$ of the matrix equation
$A_1X^2+A_0X+A_{-1}=X$,
associated with double QBDs \cite{lr:book} describing random walks in
the quarter plane, where the coefficients $A_i$ are nonnegative matrices in
$\QTD$ whose Toeplitz component is tridiagonal. This
class of problems covers a wide variety of two-queue models with
various service policies as non-preemptive priority, $K$-limited
service, server vacation and server setup \cite{ozawa}. Models of this
kind concern, for instance, bi-lingual call centers \cite{shl},
generalized two-node Jackson networks \cite{ozawa_kob18}, two-demand
models \cite{flatto}, two-stage inventory queues \cite{HZL}, and more.
Computing the minimal nonnegative solution $G$ of this matrix equation
is a fundamental step to solve the QBD by means of the matrix analytic
approach of \cite{neuts:book}. We refer the reader to the books
\cite{bini2005numerical}, \cite{lr:book}  for more details in this regard.  In
particular, we provide general conditions on the transition probabilities of the
random walk in order that $G\in\QTD$ or
$G\in\mathcal{EQT}$.

Finally, we perform an extensive numerical experimentation to show the effectiveness of our framework. We apply our approach for computing the steady state distribution of a 1-dimensional random walk with reset, for solving a quadratic matrix equation arising in a two-node Jackson network with possible breakdown and in a 2-dimensional random walk with reset.

The paper is organized as follows. In Section~\ref{sec:qt} we
introduce and analyze the classes $\QTD$ and
$\mathcal{EQT}$. 
In Section
\ref{sec:qbd} we study double QBDs which model random walks in the
quarter plane where the matrices $A_i$, for $i=-1,0,1$, are tridiagonal
quasi-Toeplitz.
Relying on the classes $\QTD$ and $\mathcal{EQT}$, we prove that the
matrix $G$ can be
written as $G=T(g)+E_g$ where $E_g$ has bounded infinity norm and $T(g)$ is the Toeplitz matrix associated with the function $g(z)$
which  solves a suitable
scalar quadratic equation.  We give sufficient conditions under which
the solution $G$ belongs to
$\QTD$ or to $\mathcal{EQT}$. Therefore,  one can plug known available algorithms --- valid for
finite matrices --- into our proposed computational framework, to approximate $G$.  Finally, in Section~\ref{sec:numerical} we test
the computational framework on some representative problems, and in Section~\ref{sec:conc} we draw the conclusions.

\section{$\mathcal {QT}$ matrices}
\label{sec:qt}
We denote by $\ell^p$, with $1 \leq p \leq \infty$, the usual Banach
space of $p$-summable sequences $x=(x_j)_{j\in\mathbb Z^+}$, with the
norms $\|x\|_p:=(\sum_{j = 1}^\infty |x_j|^p)^{\frac 1p}$ for $1\le p<\infty$, $\|x\|_\infty:= \sup_{j} |x_j|$,
and by $\Endo{\ell^p}$ the set of bounded linear operators from
$\ell^p$ into itself with the operator norm $\|A\|_p=\sup_{\|x\|_p=1}
\|Ax\|_p$. A sequence $x$ will be also referred to as a semi-infinite
vector, or simply a vector.  Moreover, we denote by
$\K(\ell^p)\subset\Endo{\ell^p}$ the subset formed by compact
operators, and by $e=(1,1,\ldots)^T\in\ell^\infty$ the vector of all
ones. Throughout this work, we will only consider operators
that can be represented as matrices with respect to the standard basis
$\{ e_i \}_{i \in \mathbb N}$. This restricts the focus on operators
that act on (and whose image is contained in) the closure of such set,
which is smaller than the entire space when $p = \infty$, since
$\ell^\infty$ is not separable.
  
The Wiener class $\mathcal W$ is the set of Laurent series
$a(z)=\sum_{i\in\mathbb Z}a_iz^i$ such that
$\|a\|\w:=\sum_{i\in\mathbb
  Z}|a_i|$ is finite. This set, which contains complex valued
functions defined on the unit circle, is a Banach algebra
\cite{bottcher2005spectral} with the norm $\norm{\cdot}\w$.
The map that associates a function
$a(z)\in\mathcal W$, called symbol, with the semi-infinite Toeplitz
matrix $T(a)=(t_{i,j})_{i,j\in\mathbb Z^+}$, $t_{i,j}=a_{j-i}$, is a
bijection between $\mathcal W$ and the set of bounded Toeplitz
operators on $\Endo{\ell^p}$ for $ p=1,\infty$.

In \cite{bmm}, a new class of semi-infinite matrices is introduced,
denoted by \QT, and is defined as the set of matrices that can be
written as the sum of a (semi-infinite) Toeplitz matrix $T(a)$ such
that $a'(z)=\sum_{i\in\mathbb Z}ia_iz^i \in \mathcal{W}$ and a correction $E = 
(E_{i,j})_{i,j \in \mathbb Z^+}$
such that
$\sum_{i,j\in\mathbb Z^+}|E_{i,j}|$ is finite.  The class \QT\ is
endowed with an appropriate norm, which makes it a Banach
algebra. This norm is denoted by $\norm{\cdot}\qt$ and is defined
as follows:
$
\norm{T(a) + E}\qt = \norm{a}\w + \norm{a'}\w + \norm{E}\f$,
 $ \norm{E}\f := \sum_{i,j\in\mathbb Z^+} |E_{ij}|$.
Observe that this norm is well-defined since both $a(z)$ and $a'(z)$ belong to $\mathcal W$.

This framework has shown to be very effective in the development of
numerical algorithms that treat the infinite dimensional case
``directly'', without the need of truncating matrices to finite
size. It provides a practical tool for solving computational problems
like computing matrix functions and solving matrix equations where the
input is given by $\QT$ matrices. We refer the reader to
\cite{bmmr,bmr,bmmj,bm:nm, bini2017functions,
  robol2019rational}
for some examples where this arithmetic has been used numerically to
solve various kinds of tasks.
However, several aspects of the theory are not yet completely
satisfactory. For instance, the requirement that the symbol $a'(z)$
lives in $\mathcal W$ is stronger than simply requiring $a(z) \in
\mathcal W$, and seems artificial.  Moreover, there are cases in the
setting of Markov chains that fit very naturally in the set of
low-rank perturbations of semi-infinite Toeplitz matrices, but cannot
be described under this framework because the correction $E$ does not
have finite norm when considering $\norm{\cdot}\f$.  A
couple of examples are given in Section~\ref{sec:numerical}.

 The aim of this section is introducing a superset of $\QT$
that allows to treat such cases maintaining the features needed to
establish a computational framework. Let us first introduce  some notation.
Given $a(z)\in\mathcal W$ define
$a^+(z)=\sum_{i\in\mathbb Z^+}a_iz^i$, $a^-(z)=\sum_{i\in\mathbb
	Z^+}a_{-i}z^i$ so that $a(z)=a_0+a^-(z^{-1})+a^+(z)$, and associate
with $a^\pm(z)$ the following semi-infinite Hankel matrices
$H(a^+)=(a_{i+j-1})_{i,j \in \mathbb Z^+}$, $H(a^-)=(a_{-i-j+1})_{i,j \in \mathbb Z^+}$.
 The following result from 
\cite[Proposition~1.3]{bottcher2005spectral} links semi-infinite Toeplitz and Hankel matrices.

\begin{theorem}[Gohberg-Feldman] \label{thm1}
  If $a(z)\in\mathcal W$, 
 then $\norm{T(a)}_p\le\norm{a}\w$,
        $\norm{H(a^-)}_p\le\norm{a}\w$,
        $\norm{H(a^+)}_p\le\norm{a}\w$. If
	$c(z)=a(z)b(z)$ where $a(z),b(z)\in\mathcal W$, then
	$
	T(a)T(b)=T(c)-H(a^-)H(b^+)$.
\end{theorem}
The Hankel matrices 
$H(a^-)$ and $H(b^+)$ are compact operators in $\Endo{\ell^p}$
for every $1 \leq p \leq \infty$ \cite[Proposition~1.2]{bottcher2005spectral}. 

\subsection{The class of $\QT_p$ matrices} \label{sec:qtp}
A more general approach for defining the set of quasi-Toeplitz
matrices is avoiding the norm $\|\cdot\|\qt$ and keeping the induced
operator norm $\|\cdot\|_p$.

\begin{definition}\label{def:qt}
Given an integer $p$, $1 \leq p \leq \infty$, we say that the
semi-infinite matrix $A$ is \emph{$p$-Quasi-Toeplitz} if it can be
written in the form
$
A=T(a)+E,
$
where $a(z)\in\mathcal W$,
and $E$ defines a compact operator in $\Endo{\ell^p}$. We refer to
$T(a)$ as the Toeplitz part of $A$, and to $E$ as the correction.  We
denote the set of $p$-Quasi-Toeplitz matrices as $\QT_p$.
\end{definition}
The set
$\QT_p$ is closed under product. In
fact, denoting $A=T(a)+E_a$, $B=T(b)+E_b$ in $\QT_p$ one has
$
C=AB=T(a)T(b)+T(a)E_b+E_aT(b)+E_aE_b$.
Moreover, denoting $c(z)=a(z)b(z)$, since in view of Theorem \ref{thm1} we have
$T(a)T(b)=T(c)-H(a^-)H(b^+)$, then it follows that
\begin{align*}
&C=T(c)+E_c,\\
&E_c=-H(a^-)H(b^+)+T(a)E_b+E_aT(b)+E_aE_b.
\end{align*}
The matrix $E_c$ is compact in $\Endo{\ell^p}$ since each addend  is the product of two operators, at least one of the two being compact in $\Endo{\ell^p}$. This proves that $\QT_p$
is closed under matrix multiplication, and being a 
subspace of $\Endo{\ell^p}$, we have the following. 

\begin{theorem}
	The class $\QT_p$ for any integer $p$,
	$1 \leq p \leq \infty$ is an algebra in $\Endo{\ell^p}$. 
      \end{theorem}

      \begin{remark}
  The set $\QT_p$ is not necessarily
  topologically closed for
  $1 < p < \infty$; for instance, for $p = 2$ it is known that
  $\norm{T(a)}_2 = \norm{a}_\infty$ \cite{bottcher2005spectral}, where $\norm{a}_\infty$ is intended as
  the sup-norm of continuous function defined for $|z|=1$.
  By the Du Bois-Reymond theorem \cite{du1873ueber}
  there exists a continuous function $a$ whose Fourier series is not
  summable. The latter could be approximated uniformly with
  polynomials in view of Weierstrass' theorem, and this produces a sequence
  of operators $T(a_n) \to T(a)$ in the $2$-norm --- but whose limit
  has symbol outside the Wiener class. In Section~\ref{sec:eqt} we show that
  for the case $p = \infty$, which is the one of interest for our applications,
  the set $\QT_\infty$ is  a (closed) Banach algebra. 
\end{remark}

The following result ensures that the set $\QT_p$  extends $\QT$.

\begin{lemma}
	For any integer $1\le p\le\infty$, it holds $\QT\subset\QT_p$. 
	\end{lemma}
	\begin{proof}
Let $A=T(a)+E\in\QT$. It is sufficient to prove that
$\norm{E}\f\ge\norm{E}_p$ for any $p\in[1,\infty]$.  Without loss of
generality we can consider the case $\norm{E}\f= 1$ so that
$|E_{ij}|\leq 1\ \forall i,j$.  In fact, if $\norm{E}\f=\theta\ne 1$,
the condition $\norm{E}\f\le\norm{E}_p$ is equivalent to
$\norm{\theta^{-1}E}\f\le\norm{\theta^{-1}E}_p$, that is,
$\norm{\widetilde E}\f\le\norm{\widetilde E}_p$ where $\widetilde
E=\theta^{-1} E$ is such that $\norm{\widetilde E}\f=1$.
Let $x$ be such that $\norm{x}_p=1$, $y=Ex$ so that
$\norm{y}_p\le\norm{E}_p$. Observe that $|x_i|\leq 1$ for any $i$ so
that
		$
		|y_i|\leq\sum_{j\geq 1}|E_{ij}x_j|\leq \sum_{j\geq 1}|E_{ij}|\leq 1
		$.
Since $p\geq 1$, then
		\[
		|y_i|^p\leq|y_i|\leq \sum_{j\geq 1}|E_{ij}|\quad
                \Rightarrow\quad \norm{y}_p\leq \norm{E}\f^{1/p}=\norm{E}\f,
		\]
where the last equality holds since $\norm{E}\f=1$. This
way, $\norm{E}_p=\sup_{\|x\|_p=1}\norm{Ex}_p\leq \norm{E}\f$.
\end{proof}
        It can be shown that the inclusion is strict.
%
%

Matrices in the $\QT_p$ class, for $p\ne 1,\infty$, can be
approximated to any arbitrary precision by using a finite number of
parameters, in the following sense.

\begin{lemma} \label{lem:decay-p}
Let $A=T(a)+E\in\QT_p$ for some integer $p\in (1,\infty)$, 
 then,
for any $ \epsilon>0$ there exist $\widetilde E\in\mathcal K(\ell^p)$ with
finite support and a Laurent polynomial $\widetilde a(z)$ such that
$\norm{A-\widetilde A}_p \leq \epsilon$ where $\widetilde
A=T(\widetilde a)+\widetilde E$.
\end{lemma}
		
\begin{proof}
	Since $a(z) \in \mathcal W$, there exists a Laurent polynomial
        $\tilde a(z)$ such that $\norm{a - \tilde a}\w \leq
        \frac{\epsilon}{2}$, and therefore, 
        $\norm{T(a) - T(\tilde
          a)}_p \leq \|a-\tilde a\|\w
         \leq\frac{\epsilon}{2}$.  Since $E$ is compact and since
        $\ell^p$ for $1 \leq p < \infty$ admits a Schauder basis,
        finite rank operators are dense in $\K(\ell^p)$, see
        \cite[Theorem 4.1.33]{Megginson}. Therefore, 
        we can find $\widehat E$ of finite rank
        $k$ such that $\norm{E - \widehat E}_p \leq
        \frac{\epsilon}{4}$. Thus, we can write $\widehat E = \sum_{j
          = 1}^k u_j v_j^T$, with $u_j \in \ell^p$ and $v_j \in
        \ell^q$, with $\frac{1}{p} + \frac{1}{q} = 1$, and $p,q >
        1$. This implies that each $u_j, v_j$ can be approximated
        arbitrarily well with vectors of finite support $\tilde u_j,
        \tilde v_j$ such that $\norm{u_j v_j^T - \tilde u_j \tilde
          v_j^T} \leq \frac{\epsilon}{4k}$.
        Setting $\tilde E := \sum_{j = 1}^k \widetilde u_j \widetilde v_j^T$,
        which has finite support, concludes the proof.
\end{proof}
		
\subsection{The class $\QTD$} \label{sec:newclass}
Observe that Lemma~\ref{lem:decay-p} does not hold for $p=1$ and for
$p=\infty$. In fact, for any random vector with components in modulus
less than 1 we have $ve_1^T\in\QT_\infty$ and
$e_1v^T\in\mathcal{QT}_1$. On the other hand, $v$ cannot be
approximated to any precision with a finite number of parameters.
This limitation is a serious drawback from the computational point of
view especially for $p=\infty$ since the $\ell^\infty$ environment is
the natural setting for Markov chains.

For this reason, we introduce a slightly different definition for
the case $p = \infty$; the case $p = 1$ can be treated by
considering the transpose matrix\footnote{Note that, even if $\ell^1$ is much
  smaller of the dual of $\ell^\infty$, the additional constrain
  of considering operators representable as matrices over the
  canonical basis, implies $\QT_1 = (\QT_\infty)^\star$.}  of elements
in $\QT_\infty$.

\begin{definition}\label{def:decay}
  A matrix $E \in \Endo{\ell^\infty}$ has the \emph{decay property} if
  the vector $w:=|E| e$, $w=(w_i)_{i\in\mathbb Z^+}$, 
   is such that
  $\lim_{i\to\infty}w_i=0$, where $|E|:=(|E_{i,j}|)_{i,j \in \mathbb Z^+}$.
\end{definition} 

\begin{definition}\label{def:qt-inf}
  We define $\QTD$ the class of all the matrices which can be written
  in the form $ A=T(a)+E, $ where $a(z)\in\mathcal W$ and
  $E\in\Endo{\ell^\infty}$ has the decay property. The superscript ``$d$'' denotes ``decay''.
\end{definition}

The decay property allows to state an approximability result in the
same spirit of Lemma~\ref{lem:decay-p} for matrices in
$\Endo{\ell^\infty}$.

\begin{lemma}\label{th:usfl}
Let $E\in \Endo{\ell^\infty}$, and let $E^{(k)}$ be the matrix that 
coincides with $E$ in the
leading principal $k\times k$ submatrix and is zero elsewhere.
Then, the following are equivalent:
\begin{enumerate}[(i)]
	\item $E$ has the decay property;
	\item $\lim_{k \to \infty} \norm{E - E^{(k)}}_\infty = 0$.
\end{enumerate}
In particular, if $E$ has the decay property,
 then it represents a
compact operator in $\Endo{\ell^\infty}$.
\end{lemma}
\begin{proof}
	We first prove $(i) \implies (ii)$. 
 Since $w=|E| e$ is such that
$\lim_{i} w_i=0$, then for any $\epsilon>0$ there exists $m$
such that $w_i\le\epsilon$ for any $i> m$. Therefore, the matrix $E^{(m)}$
is such that the vector $v=|E-E^{(m)}|e$ has components $v_i\le
\epsilon$ for $i> m$. On the other hand, since $|E|\in
\Endo{\ell^\infty}$, then each row $r^{(i)}=e_i^T |E|$ 
has sum of its
entries finite, therefore, 
 there exists $n_i$ such that
$\sum_{j=n_i+1}^\infty r^{(i)}_j\le\epsilon$.  Setting
$n=\max\{m,n_1,n_2,\ldots,n_m\}$ yields $\norm{E-E^{(k)}}_\infty\le\epsilon$
for any $k\ge n$.  
Concerning $(ii) \implies (i)$, we consider $v^{(k)}=|E-E^{(k)}|e$,
and $ w=|E|e$. Observe that, since $E^{(k)}_{i,j}=0$ for $i> k$ or for
$j>k$, then $v^{(k)}_i=w_i$ for $i> k$. Moreover, since $\norm{
  v^{(k)}}_\infty =\norm{E-E^{(k)}}_\infty$, then $\lim_k
\norm{v^{(k)}}_\infty =\lim_k\norm{E-E^{(k)}}_\infty=0$ so that for
any $\epsilon>0$ there exists $k_0$ such that
$\norm{v^{(k)}}_\infty\le \epsilon$ for any $k\ge k_0$, whence $
v^{(k)}_i\le\epsilon$ for any $i$.  In particular,
$v_i^{(k_0)}\le\epsilon$ for any $i$.  Thus, since $w_i=v^{(k_0)}_i$
for any $i> {k_0}$, then $w_i\le \epsilon$ for any $i> {k_0}$.
Finally, since $E$ is the limit of compact operators it is compact.
\end{proof}

An immediate consequence of Lemma~\ref{th:usfl} is that any $A\in\QTD$ can be
approximated by a finitely representable matrix in
$\QTD$ as stated in the following corollary.

\begin{corollary}\label{cor:usfl}
	Let $A = T(a) + E \in \QTD$. Then, for every $\epsilon > 0$
        there exists a Laurent polynomial $\tilde a(z)$ and an integer
        $k$ such that $\norm{A - T(\tilde a) - E^{(k)}}_\infty \leq
        \epsilon$.
\end{corollary}

The class of matrices having the decay property is closed as specified
by the following

\begin{theorem}\label{thm:decayclosure}
Let $E_k\in \Endo{\ell^\infty}$, for $k\in\mathbb Z^+$, have the decay
property. Assume that there exists $E\in \Endo{\ell^\infty}$ such that
$\lim_k\norm{E_k-E}_\infty=0$. Then $E$ has the decay property as
well.
\end{theorem}
\begin{proof}
  It is enough to prove that $\lim_i v_i=0$ for $v=|E|e$. Denote
  $v^{(k)}=|E_k|e$. From $|E_k-E|\ge ||E_k|-|E||$ we deduce that
  $|E_k-E|e\ge ||E_k|-|E||e\ge|v^{(k)}-v|$. Whence
  $\norm{E_k-E}_\infty=\norm{|E_k-E|e }_\infty\ge
  \norm{v^{(k)}-v}_\infty$. This implies that $\lim_k \sup_i
  |v^{(k)}_i-v_i|=0$. We now deduce that $\lim_i v_i=0$.  From the
  condition $\lim_k \sup_i |v^{(k)}_i-v_i|=0$ we find that for any
  $\epsilon>0$ there exists $k_0$ such that $ \sup_i
  |v^{(k)}_i-v_i|\le\epsilon$ for any $k\ge k_0$, that is
  $|v^{(k)}_i-v_i|\le\epsilon$ for any $i$ and for $k\ge k_0$.
  Therefore, 
  $v_i\in[v_i^{(k)}-\epsilon, v_i^{(k)}+\epsilon]$ for any
  $i$ and for any $k\ge k_0$. On the other hand from the condition
  $\lim_i v^{(k)}_i=0$ for any $k$ we deduce that for any $\epsilon>0$
  and for any $k$ there exists $i_k$ such that
  $|v_i^{(k)}|\le\epsilon$ for any $i\ge i_k$. Combining the two
  properties yields $v_i\in[-2\epsilon,2\epsilon]$ for any $i\ge
  i_{k_0}$. That is $\lim_i v_i=0$.
\end{proof}

We consider the quotient
space of $\Endo{\ell^\infty}$ 
under the equivalence relation:
$A \doteq B$ if and only if $A - B$ has the decay property. If 
$A$ is representable with a finite number of parameters 
then, in light of Lemma~\ref{th:usfl}, every $B$ such that 
$A \doteq B$ is also representable using a finite number of
parameters. Matrices with the decay property form a right ideal.
\begin{lemma}\label{lem:ideal}
	Let $A,B\in\Endo{\ell^\infty}$ such that $A\doteq 0$. Then
\begin{enumerate}[(i)]
\item if $B\doteq 0$,
      then $A+B\doteq 0$,
\item $A  B \doteq 0$,
\item if $B=T(b)$ with $b\in\mathcal W$, 
      then $B A\doteq 0$.
\end{enumerate}
	\end{lemma}
\begin{proof}
	Claim $(i)$ easily follows applying the definition.
        Concerning $(ii)$, we notice that $|AB|e\leq |A||B|e\leq
        \norm{B}_\infty|A|e$ which is an infinitesimal vector.  Let
        $w=|A|e$ with entries $w_i$ such that
        $\lim_{i\to\infty}w_i=0$.  In order to prove $(iii)$, let us
        start by considering $B=T(b)$ where the symbol $b$ has finite
        support, more precisely $b_j=0$ whenever $|j|>k$, for some
        $k\in\mathbb N$. Then we have $|BA|e\leq |B|w=g$ whose entries
        $g_i$ verify $g_i=\sum_{j=i-k}^{i+k}|b_{j-i}||w_j|$ for
        $i>k$. Therefore, 
        $g_i\to 0$. If $b$ has not finite support we
        consider $b_k$ the Laurent polynomial obtained by truncating
        $b$ with coefficients in the exponent range $[-k,k]$; clearly
        $\norm{T(b)-T(b_k)}_\infty\to 0$ which implies
        $\norm{T(b)A-T(b_k)A}_\infty\to 0$. Hence, the claim follows
        applying Theorem~\ref{thm:decayclosure}.
	\end{proof}
Note that, $A\doteq 0\centernot\implies B A\doteq 0$; indeed consider
$A=e_1e_1^T$ and $B=e e_1^T$ as a counterexample.

We shall now prove that the Hankel matrices arising
in Theorem \ref{thm1} have the decay property. 
\begin{lemma}
  Let $a(z)\in\mathcal W$, 
       then  $H(a^-)\doteq 0$ and $H(a^+)\doteq 0$.
  	\end{lemma}
\begin{proof}
	Consider the vector $w=|H(a^-)|e$; it holds that
	$w_i=\sum_{j=i}^{\infty}|a_{-j}|$, whence $\lim_i w_i=0$, i.e.,
	$H(a^-)\doteq 0$. The same holds for $H(a^+)$.
	\end{proof}
This, combined with Lemma~\ref{lem:ideal}, yields the following Corollary.
\begin{corollary}\label{cor:T(ab)}
	Let $a,b\in\mathcal W$, 
        then
	\begin{equation}\label{eq:dec}
	\begin{aligned}
	&T(a)T(b)\doteq T(ab)\doteq T(b)T(a),\\
	&T(a)T(a^{-1})\doteq I,\quad \hbox{if }a(z)\ne 0\hbox{ for }|z|=1.
	\end{aligned}
	\end{equation}
	\end{corollary}

The next result will be crucial for proving the closedness of  $\QTD$.
\begin{lemma}\label{lem:norm}
	If $A\in\QTD$, $A=T(a)+E$, then $\norm{A}_\infty\ge \norm{a}_\w$.
\end{lemma}
\begin{proof}
	We prove that for any $\epsilon>0$ there exists $i_0$ such
        that for any $i\ge i_0$ we have $e_i^T|A|e\ge
        \norm{a}_\w-2\epsilon$. Since $\norm{A}_\infty=\sup_i
        e_i^T|A|e$, then from the latter inequality it follows that
        $\norm{A}_\infty\ge \norm{a}_\w$.  In order to prove the claim,
        we observe that since $|A|\ge |T(a)|-|E|$ we have
	$
	e_i^T|A|e\ge e_i^T|T(a)|e-e_i^T|E|e$.
	From the decay property of $E$ we have that there exists $h_0$ such
	that for any $i\ge h_0$ we have $e_i|E|e\le\epsilon$. On the
	other hand, since
	$e_i^T|T(a)|=\sum_{j=-i}^\infty|a_{j}|=\norm{a}_\w-\sum_{j=-\infty}^{-i-1}|a_j|$,
	and since $a(z)\in\mathcal W$, 
        then there exists $k_0$ such that
	$e_i^T|T(a)|=\norm{a}\w-\epsilon_i$, where $|\epsilon_i|\le \epsilon$
	for any $i\ge k_0$. Thus for any $i\ge i_0=\max\{h_0,k_0\}$ we have
	$
	e_i^T|A|e\ge \norm{a}\w-|\epsilon_i|-\epsilon\ge \norm{a}\w-2\epsilon
	$.
\end{proof}
\begin{theorem}\label{thm:banach}
	The class $\QTD$ is a Banach algebra with the
	infinity norm.
\end{theorem}
\begin{proof} 
	For the property of algebra it is enough to show that if
        $A=T(a)+E_a$, $B=T(b)+E_b$ are in $\QTD$, then also $A+B$,
        $\alpha A$ and $A B$ are in $\QTD$.  For the first two
        matrices the property is trivial since $\alpha E_a$ and
        $E_a+E_b$ have the decay property. For the third condition,
        Lemma~\ref{lem:ideal} and Corollary~\ref{cor:T(ab)} imply
        $AB\doteq T(ab)$.
	It remains to prove that $\QTD$ is complete. If
	$X_k=T(x_k)+E_k\in\QTD$, $k\ge 0$, is a Cauchy sequence with the infinity
	norm, then, since $ \Endo{\ell^\infty}$ is a Banach space there
	exists $X\in\Endo{\ell^\infty}$ such that
	$\lim_k\norm{X_k-X}_\infty=0$. We have to prove that
	$X\in\QTD$, i.e., $X=T(x)+E$ for some
	$x(z)\in\mathcal W$ and $E\in\Endo{\ell^\infty}$ with the decay
	property. From Lemma \ref{lem:norm} we have
	$\norm{X_k-X_h}_\infty\ge\norm{x_k-x_h}\w$ therefore, since $\{X_k\}_k$ is
	Cauchy, then also $\{x_k(z)\}_k$ is Cauchy with the Wiener norm. Thus,
	since $\mathcal W$ is a Banach space, %
        then there exists
	$x(z)\in\mathcal W$ such that $\lim_k\norm{x_k(z)-x(z)}\w=0$. Now
	consider $E_k-E_h$. Since $E_k-E_h=X_k-X_h+T(x_k-x_h)$ we have
	$\norm{E_k-E_h}_\infty\le\norm{X_k-X_h}_\infty+\norm{x_k-x_h}\w$,
        whence $\{E_k\}_k$
	is Cauchy in $\Endo{\ell^\infty}$ therefore, 
        there exists $E\in
	\Endo{\ell^\infty}$ such that $\lim_k \norm{E_k-E}_\infty=0$.
        It remains to prove
	that $E$ has the decay property.  This follows from
        Theorem~\ref{thm:decayclosure}.
\end{proof}

\subsection{The class $\EQT$}\label{sec:eqt}
The matrices modeling stochastic processes with
restarts do not belong to $\QTD$. Indeed, they belong to $\QTD$ up to a correction part whose columns do not decay to $0$, but instead converge to a nonzero
limit. In particular, the correction does not have the decay property
but it is still (approximately) representable by a finite set of
parameters. 
In this section we introduce an appropriate extension of
$\QTD$.
\begin{definition}\label{def:eqt}
	We say that the semi-infinite matrix $A$ is
        \emph{extended-quasi-Toeplitz} if it can be written in the
        form
	\begin{equation}\label{eq:eqt}
	A=T(a)+E+ev^T,
	\end{equation}
	where $a(z)\in\mathcal W$, $E\doteq 0$ and $v\in\ell^1$.  We
        denote the set of extended-quasi-Toeplitz matrices with the
        symbol $\EQT$.
\end{definition}
Clearly, $\QTD\subset\EQT\subset\Endo{\ell^\infty}$, and in view of
Corollary \ref{cor:usfl} the matrices in these classes are representable
with a finite number of parameters within a given error bound
$\epsilon$. Indeed, the term $ev^T$ in \eqref{eq:eqt} can be
approximated --- in the $\infty$-norm --- by truncating $v\in\ell^1$
to a vector of finite support. Similarly to $\QTD$, the set $\EQT$ is
a Banach algebra.  It is immediate to check that $A, B \in \EQT
\implies A + B \in \EQT$. Multiplication requires some
explicit computations.
\begin{lemma}\label{lem:eqt-mult}
Let $A=T(a)+E_a+ev_a^T$ and $B=T(b)+E_b+ev_b^T$ be matrices in
$\EQT$. Then $C=A B\in\EQT$ and $C=T(c)+E_c+ev_c^T$ where $c=ab$,
$v_c=\left(\sum_{j\in\mathbb Z}a_j\right)v_b+B^Tv_a$ and
\[ E_c=T(a)E_b+E_aT(b)-H(a^-)H(b^+)+E_aE_b+(E_a-H(a^-))ev_b^T.\]
\end{lemma}
\begin{proof}
	The result follows via a direct computation using the relation
        $T(a)e= \left(\sum_j a_j\right) e-H(a^-)e$. Note that,
        $E_c\doteq 0$ in view of Lemma~\ref{lem:ideal}.
	\end{proof}
In order to state the main result of this section, we need the
following generalization of Lemma~\ref{lem:norm}.

\begin{lemma}\label{lem:norm1}
If $A\in\EQT$, $A=T(a)+E+ev^T$, then $\norm{A}_\infty\ge \norm{a}\w+\norm{v}_1$.
\end{lemma}
\begin{proof}
 We prove that for any $\epsilon>0$ there exists $k$ such that
 $\norm{e_k^TA}_1\ge \norm{a}\w+\norm{v}_1-5\epsilon$ so that the
 claim follows from the inequality $\norm{A}_\infty\ge
 \norm{e_k^TA}_1$ and by the arbitrarity of $\epsilon$. To this end,
 given $\epsilon$, it is sufficient to choose $k=2p+1$ where $p$ is
 large enough so that $\sum_{i=p+1}^\infty |v_i|\le\epsilon$,
 $\sum_{i=-\infty}^{-p-1}|a_i|\le\epsilon$ and $w_k\le\epsilon$ where
 $w=|E|e$.  This way the $k$th row of $A$ is $r_k=e_k^TA=v^T+u^T+s^T$
 where $u^T=[a_{-2p},a_{-2p+1},\ldots]$, $s^T=e_k^TE$. Observe that
 $\norm{s}_1=w_k\le\epsilon$ so that
\begin{equation}\label{eq:lemnr0}
\norm{r_k}_1\ge \norm{v+u}_1-\epsilon.
\end{equation} In order to estimate $\norm{v+u}_1$,
decompose $v$ as $v=\tilde v+\hat v$ where $\tilde
v=[v_1,\ldots,v_p,0,\ldots]^T$, $\hat
v=[0,\ldots,0,v_{p+1},\ldots]^T$. Do the same with $u=\tilde u+\hat
u$. Since $\tilde v$ and $\hat v$ have disjoint supports, then
$\norm{\tilde v+\hat u}_1=\norm{\tilde v}_1+\norm{\hat u}_1$,
moreover, thanks to the choice of $p$, we have $\norm{\hat v+\tilde
  u}_1\le 2\epsilon$.  Thus, we deduce that
\begin{equation}\label{eq:lemnr}
  \norm{v+u}_1\ge \norm{\tilde v+\hat u}_1-\norm{\hat v+\tilde u}_1\ge
  \norm{\tilde v}_1+\norm{\hat u}_1-2\epsilon.
\end{equation}
 Finally, since $\tilde v=v-\hat v$ we deduce that $\norm{\tilde
   v}_1\ge \norm{v}_1-\epsilon$, and similarly, $\norm{\hat
   u}_1\ge\norm{u}_1-\epsilon$.  Combining the latter two inequalities
 with \eqref{eq:lemnr0} and \eqref{eq:lemnr}, yields
$
\norm{r_k}_1\ge\norm{v+u}_1-\epsilon\ge\norm{\tilde v+\hat u}_1-5\epsilon
$
which completes the proof.
\end{proof}
\begin{remark}\label{rem:uniqueness}
Lemma~\ref{lem:norm1} allows to easily show the uniqueness of the
decomposition of an element in \EQT. Indeed, suppose there exist two
different representations of the same matrix
$A=T(a)+E_a+ev_a^T=T(a')+E_{a'}+ev_{a'}^T$. Then
\[
0=\norm{A-A}_\infty\geq \norm{a-a'}\w+\norm{v_a-v_{a'}}_1\implies
a\equiv a',\quad v_a= v_{a'}. 
\]
By difference, we finally get $E_a=E_{a'}$.
\end{remark}
\begin{theorem}
The class ${\EQT}$ is a Banach algebra with the infinity norm.
\end{theorem}
\begin{proof}
The class is clearly closed under addition and multiplication by a
scalar. Moreover, it is closed under multiplication in view of
Lemma~\ref{lem:eqt-mult}. In order to prove that it is a Banach space,
it is sufficient to follow the same argument used in the proof of
Theorem \ref{thm:banach} relying on Lemma \ref{lem:norm1}.
\end{proof}

\subsection{Extended \cqttoolbox}\label{sec:eqt-tool}
Here, we describe how the computational framework for $\EQT$ has been
implemented on top of \cqttoolbox~\cite{bmr}. The latest release of
the software includes this tool.

A matrix $A\in\EQT$ is represented relying on the unique decomposition
(see Remark~\ref{rem:uniqueness}) $A=T(a)+E+ev^T$. The terms $T(a)$
and $E$ are represented using the same data structures as the
$\QT_{\infty}$ class. This is possible because the entries of $E\doteq
0$ allows to truncate it to its top-left corner. The format is
extended by storing a truncation $\widetilde v$ of the vector
$v\in\ell^1$. This is performed by requiring $\norm{v-\widetilde
  v}_1\leq\epsilon\norm{A}_{\infty}$. As illustrative example, we report the Matlab code that define the matrix $A_0$ of the Jackson network with reset introduced in Section~\ref{sec:tandem}.

\begin{lstlisting}[mathescape=true]
>> E = $\gamma$ * $\mu_1$ + $\gamma$ - 1;
>> pos = [0 $\lambda_1$]; 
>> neg = [0 $\gamma$ * (1 - p) * $\mu_1$];
>> v = 1 - $\gamma$;
>> A0 = cqt('extended', neg, pos, E, v);
\end{lstlisting}
The arithmetic operations in the class $\EQT$ can be performed by using the standard Matlab arithmetic operators {\tt +,-,*,/,$\tt \backslash$} and the operator {\tt inv}.

We conclude the section by summarizing the relations that link the parameters defining the input
of a matrix operation 
 to those of its outcome. Some of them have
been already presented in Section~\ref{sec:eqt}, the others can be
verified via a direct computation.  In what follows we consider two
$\EQT$ matrices $A=T(a)+E_a+ev_a^T$ and $B=T(b)+E_b+ev_b^T$.

\begin{description}
	\item[Addition] If $C=A+B$, 
        then
	$C=T(a+b)+ E_c + e(v_a+v_b)^T,\qquad E_c=E_a+E_b$.
	\item[Multiplication] If $C=AB$, 
        then
	\begin{align*}
	C&=T(ab)+ E_c + e(s_a
        v_b+B^Tv_a)^T,\qquad  s_a=\sum_{j\in\mathbb
          Z}a_j\\ E_c&=T(a)E_b+E_aT(b)-H(a^-)H(b^+)+E_aE_b+
        (E_a-H(a^-))ev_b^T.
	\end{align*}
	\item[Inversion] The inversion formula is obtained by means of
          the Woodbury identity, considering an $\EQT$ matrix as a rank
          one correction of a $\QTD$ matrix. If $C=A^{-1}$,
          then
	\[
	C= (T(a)+E_a)^{-1}- (T(a)+E_a)^{-1} ev_a^T(T(a)+E_a)^{-1}/ (1+ v_a^T(T(a)+E_a)^{-1}e).
	\]
\end{description}
In this equation, although the terms are not
separated as in the other expressions, all the operations involved are
performed with the addition and multiplication formulas for the $\QTD$
class.

It is interesting to point out that the arithmetic introduced in the Toolbox \cqttoolbox{},
includes also the case of finite QT-matrices where the correction to the Toeplitz part involves the top leftmost and the bottom rightmost corners. This allows to deal effectively with finite matrices of large size. We refer the reader
to \cite[Section 3.5]{bmr} for further details. 

\section{Double QBDs and related random walks in the quarter plane}\label{sec:qbd}
The use of the Matrix Analytic Method of Neuts \cite{neuts:book} allows to
recast the computation of the invariant probability vector of a 
QBD process into 
determining the minimal nonnegative solution $G$ of the matrix equation
\begin{equation}\label{eq:mateqG1}
X=A_{-1}+A_0X+A_1X^2.
\end{equation}
A solution $G=(g_{i,j})_{i,j \in \mathbb Z^+}$ of a matrix equation is said to be minimal nonnegative if $g_{i,j}\ge 0$, and for any other solution 
$X=(x_{i,j})_{i,j \in \mathbb Z^+}$ such that $x_{i,j}\ge 0$ it follows $g_{i,j}\le x_{i,j}$ for any $i,j$.
In this section we consider the case where the equation has infinite
coefficients $A_{-1},A_0,A_1\in\QTD$ that originate from a random walk in the
quarter plane governed by a discrete time Markov chain. In this case, the minimal nonnegative solution $G$ exists,  and we provide
conditions under which $G$ 
belongs to $\QTD$ or to $\EQT$.  The Markov chain
describes the dynamics of a particle $p$ which can occupy the points
of a grid in the quarter plane of integer coordinates $(r,s)$, for
$r,s\ge 0$.  If $p$ occupies an inner position, i.e., if $r,s>0$, then
at each instant of time it can move to $(r+j,s+i)$ with given
probabilities $a_{i,j}$ for $i,j=-1,0,1$.  If the particle is along
the $y$ axis, i.e., if $r=0$ and $s>0$, then it can move to $(j,s+i)$
with given probability $y_{i,j}$ for $i=-1,0,1$, $j=0,1$.  Similarly,
if the particle is along the $x$ axis, i.e., if $r>0$ and $s=0$, then
it can move to $(r+j,i)$ with probability $x_{i,j}$ for $i=0,1$,
$j=-1,0,1$. Finally, if $p$ is in the origin, it can move to the
position $(j,i)$ with probability $o_{i,j}$ for $i,j=0,1$. Figure
\ref{fig:rwqp0} pictorially describes an example of random walk in the
quarter plane.

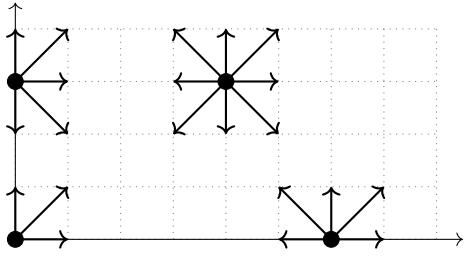
\begin{figure}
\[
  \begin{tikzpicture}[scale=0.7]
\draw[help lines][thin][dotted] (0,0) grid (8,4);
\draw [<->] 
            (0,4.5)--(0,0)--(8.5,0);

\draw [->] [thick] (4,3) -- (5,3); 
\draw [->] [thick] (4,3) -- (5,4);
\draw [->] [thick] (4,3) -- (4,4); 
\draw [->] [thick] (4,3) -- (4,2);
\draw [->] [thick] (4,3) -- (3,2); 
\draw [->] [thick] (4,3) -- (3,3);
\draw [->] [thick] (4,3) -- (3,4); 
\draw [->] [thick] (4,3) -- (5,2);

\draw [fill] (4,3) circle [radius=0.15];


\draw [->] [ thick] (0,3) -- (1,4);
\draw [->] [ thick] (0,3) -- (1,3);
\draw [->] [ thick] (0,3) -- (1,2);
\draw [->] [ thick] (0,3) -- (0,4);
\draw [->] [ thick] (0,3) -- (0,2);
\draw [fill] (0,3) circle [radius=0.15];


\draw [->] [ thick] (6,0) -- (5,0);
\draw [->] [ thick] (6,0) -- (5,1);
\draw [->] [ thick] (6,0) -- (6,1);
\draw [->] [ thick] (6,0) -- (7,1);
\draw [->] [ thick] (6,0) -- (7,0);
\draw [fill] (6,0) circle [radius=0.15];

\draw [->] [ thick] (0,0) -- (1,1);
\draw [->] [ thick] (0,0) -- (0,1);
\draw [->] [ thick] (0,0) -- (1,0);
\draw [fill] (0,0) circle [radius=0.15];


  \end{tikzpicture}
  \]
\caption{Pictorial description of a random walk in the quarter plane,
    the points of the grid which have integer coordinates $(r,s)$,
    correspond to the states of the Markov chain.
    The particle can move in the grid of only one step  inside the quarter plane with assigned probabilities.
}\label{fig:rwqp0}
\end{figure}

The Markov chain which describes this model is defined by the double
infinite set of states $(r,s)$, $r,s\ge 0$, and by the transition probability
matrix $P$ whose entry with row index $(r,s)$ and column index
$(r',s')$ provides the probability of transition from state $(r,s)$ to
state $(r',s')$ in one time unit. Due to the double indices, the matrix $P$ has a multilevel 
structure and can take a different form according to the kind of
lexicographical order which is used to sort the pairs
$(r,s)$.
Denote  $\qtoep(b_0,b_1; a_{-1},a_0,a_1)$ the quasi Toeplitz
matrix with symbol $a_{-1}z^{-1}+a_0+a_1z$ and with correction
$E=e_1(b_0-a_0,b_1-a_1,0,\ldots)$. Similarly, denote  the block
quasi Toeplitz matrix 
$\qtoep(B_0,B_1;\,A_{-1},A_0,A_1)$.
Ordering the states column-wise
 as $(r,s)$, $s=0,1,\ldots,~r=0,1,\ldots$,
yields
$P=\qtoep(B_0,B_1;\,A_{-1},A_0,A_1)$, with
$A_i=\qtoep(y_{i,0},y_{i,1};\, a_{i,-1},a_{i,0},a_{i,1})$, 
$B_i=\qtoep(o_{i,0},o_{i,1};\, x_{i,-1},x_{i,0},x_{i,1})$.
More specifically we have
\begin{equation}\label{eq:P}
P=
\begin{bmatrix}
B_0&B_1\\
A_{-1}&A_0&A_1\\
&\ddots&\ddots&\ddots
\end{bmatrix}.
\end{equation}
Ordering the states row-wise
for $r=0,1,\ldots,~s=0,1,\ldots$,
yields 
 $\wh P=\qtoep(\wh B_0,\wh B_1;\, \wh A_{-1},\wh A_0,\wh A_1)$,  with
 $\wh A_j=\qtoep(x_{0,j},x_{1,j};\, a_{-1,j},a_{0,j},a_{1,j})$,
 $\wh B_j=\qtoep(o_{0,j},o_{1,j};\, y_{-1,j},y_{0,j},y_{1,j})$.
The matrix \eqref{eq:P} defines a double QBD  
  process (DQBD) \cite{miya},
\cite{lr:book}, 
which leads to the matrix equation \eqref{eq:mateqG1}. We have a
similar equation if the row-wise ordering of the states is adopted.
We refer to the row-wise representation as the {\em flipped version}
which is obtained by exchanging the roles of the axes.

It is useful to denote 
\[
\begin{array}{ll}
x_{i,:}(z)=x_{i,-1}z^{-1}+x_{i,0}+x_{i,1}z,~ i=0,1,\quad
& x_{:,j}(w)=x_{0,j}+x_{1,j}w,\quad j=-1,0,1,\\
 y_{i,:}(z)=y_{i,0}+y_{i,1}z,\quad i=-1,0,1,
&  y_{:,j}(w)=y_{-1,j}w^{-1}+y_{0,j}+y_{1,j}w,\quad j=0,1,  \\
a_{i,:}(z)=a_{i,-1}z^{-1}+a_{i,0}+a_{i,1}z,
&  a_{:,j}(w)=a_{-1,j}w^{-1}+a_{0,j}+a_{1,j}w,
~ i,j=-1,0,1.
\end{array}
\]
 For the sake of
notational simplicity, if not differently specified, we write $a_i(z)$
in place of $a_{i,:}(z)$.  Since $a_{i,j}$ are probabilities we have
$a_{i,j}\ge 0$, $\sum_{i,j}a_{i,j}=1$, that is,
$a_{-1}(1)+a_0(1)+a_1(1)=1$.  Similarly for $x_{i,j}$, $y_{i,j}$ and
$o_{i,j}$.
Moreover, we introduce the following notation
\begin{align*}
& d_1=
a_{1,:}(1)-a_{-1,:}(1),~~d_2=a_{:,1}(1)-a_{:,-1}(1),\\
& s_1=
y_{1,:}(1)-y_{-1,:}(1),~~s_2=x_{:,1}(1)-x_{:,-1}(1),\\
& r_1=d_2 x_{1,:}(1)-d_1 s_2,~~r_2=d_1 y_{:,1}(1)-d_2 s_1.
\end{align*}

The following result of  \cite[Theorem 1.2.1]{fayolle} and
\cite[Lemma 6.4]{miyazawa2011light} provides
 a necessary and sufficient condition for the positive
recurrence  of the random walk in terms of the values of
the probabilities $a_{i,j}$, $x_{i,j}$, $y_{i,j}$.

\begin{lemma}\label{lem:pr}
Assume that $(d_1,d_2)\ne (0,0)$.
  The DQBD process is positive
recurrent if and only if one of the following conditions holds:
\begin{enumerate}
\item $d_1<0$, $d_2<0$, $r_1 <0$, $r_2<0$;
\item $d_1\ge 0$, $d_2<0$, $r_2<0$, and $s_2<0$ for $x_{1,:}(1)=0$;
\item $d_1<0$, $d_2\ge 0$, $r_1<0$ and $s_1<0$ for $y_{:,1}(1)=0$.
\end{enumerate}
\end{lemma}

 In the following, we will consider the
inequalities $A_{-1}e>A_1e$ or $A_{-1}e\ge A_1e>0$. For the structure
of the matrices $A_1$ and $A_{-1}$, this set of infinitely many
inequalities reduces just to a pair of inequalities. For instance, the
condition $ A_{-1}e>A_1e $ is equivalent to $a_{-1}(1)>a_1(1),~
y_{-1}(1)>y_1(1) $, while $ A_{-1}e\ge A_1e>0 $ is equivalent to
$a_{-1}(1)\ge a_1(1)>0,~ y_{-1}(1)\ge y_1(1)>0 $.  From the
probabilistic point of view, the above inequalities say that the
overall probability that the particle moves down is greater than the
overall probability that the particle moves up.
We observe that, according to Lemma \ref{lem:pr} if $A_{-1}e>A_1e$ and
$\widehat A_{-1}e>\widehat A_1e$, 
 then condition 1 holds. Moreover,
if the DQBD is positive recurrent, then at least one of the
conditions $a_{:,-1}(1)>a_{:,1}(1)$, $ a_{-1,:}(1)>a_{1,:}(1)$ is
satisfied.

Now, we are ready to prove the following result which gives sufficient
conditions for the stochasticity of $G$.

\begin{theorem}
If $A_{-1}e>A_1e$ the minimal nonnegative
solution $G$ of the matrix equation \eqref{eq:mateqG1} is stochastic,
i.e., $Ge=e$.
\end{theorem}
\begin{proof}
  Observe that $G$ is independent of the values $x_{i,j}$ defining
  $B_0$ and $B_1$. Therefore, it is sufficient to choose the
  probabilities $x_{i,j}$, $i=0,1$, $j=-1,0,1$ in such a way that the
  DQBD \eqref{eq:P} defined by the matrices $A_{-1},A_0,A_1$ and by
  the boundary conditions $B_0,B_1$ is positive recurrent.  In light
  of Theorem 7.1.1 of \cite{lr:book}, this implies that $Ge=e$. To
  this end, consider the DQBD \eqref{eq:P} defined by the matrices
  $A_{-1},A_0,A_1$ and by the boundary conditions $B_0,B_1$ to be
  suitably chosen. The assumption $A_{-1}e>A_1e$ implies that
  $d_1<0$. If $d_2\ge 0$, then we choose $x_{i,j}$ such that
  $r_1<0$. This way, in view of part 3 of Lemma \ref{lem:pr}, the DQBD
  is positive recurrent. On the other hand if $d_2<0$, since $s_1<0$,
  then $r_2<0$.  Concerning $r_1$, we choose $x_{i,j}$ such that
  $r_1<0$, so that, in view of part 1 of Lemma \ref{lem:pr}, the DQBD
  is positive recurrent.
\end{proof}

Consider the sequence $\{G_k\}_k$ defined by 
\begin{equation}\label{eq:3it}
\begin{aligned}
&G_0=0\\
&G_{k+1}=A_1G_k^2+A_0G_k+A_{-1},\quad k=0,1,\ldots.
\end{aligned}
\end{equation}
Since $A_{-1},A_0,A_1,G_0\in\QTD$ and since $\QTD$ is an algebra, then  
 all the matrices $G_k$ belong to $\QTD$ so that they can be written as
 $G_k=T(g_k)+E_k$. Moreover, from \eqref{eq:3it} it follows that $g_k(z)\in\mathcal W$ is a Laurent polynomial and $E_k$
 has a finite support. Observe also that, by construction, the symbols $g_k(z)$
 are such that
\begin{equation}\label{eq:gk}
g_{k+1}(z)=a_{-1}(z)+a_0(z)g_k(z)+a_1(z)g_k(z)^2,\quad g_0(z)=0.
\end{equation}
Equation \eqref{eq:gk} can be viewed as a functional relation between
Laurent polynomials in the variable $z$, and also as a point-wise
equation valid for any complex value $\zeta$ of the variable of $z$
such that $|\zeta|=1$.
It is well known \cite{lr:book} that $\{G_k\}_k$ is an increasing
sequence which converges point-wise to the minimal nonnegative
solution $G$ of the matrix equation \eqref{eq:mateqG1}.  Our aim is to
provide sufficient conditions under which the sequence $\{G_k\}_k$
converges in the infinity norm and the limit $G$ can be written in the
form $G=T(g)+E_g$.  We split this analysis into two parts: the
analysis of the sequence $g_k(z)$ and that of the correction $E_k$.

\subsection{A scalar equation}\label{sec:scaleq}
In this section we prove that the sequence $\{g_k(z)\}_k$ of Laurent
polynomials defined in \eqref{eq:gk} converges in the Wiener norm to a
fixed point $g(z)\in\mathcal W$ of \eqref{eq:gk}, we show that $g(z)$
has nonnegative coefficients, is such that $g(1)\le 1$ and for any
$z\in\mathbb C$ of modulus 1, $g(z)$ is the solution of minimum
modulus of the scalar equation $
a_1(z)\lambda^2+(a_0(z)-1)\lambda+a_{-1}(z)=0$.

We need the following notation. Given two functions
$a(z)=\sum_{i\in\mathbb Z}a_iz^i$, $b(z)=\sum_{i\in\mathbb Z}b_iz^i$,
$a(z),b(z)\in\mathcal W$ we write $a(z)\le_{cw}b(z)$ if the inequality
holds coefficient-wise, i.e., if $a_i\le b_i$ for $i\in\mathbb Z$. We
have the following result.

\begin{theorem}\label{th:g}
Under the assumption $a_{i,j}\ge 0$, $\sum_{i,j=-1}^1 a_{i,j}=1$,
there exists $g(z)\in\mathcal W$ such that $\lim_k\norm{g-g_k}_\w=0$,
where $g_k(z)$ is defined in \eqref{eq:gk}. Moreover $g(1)\le 1$,
$0\le_{cw}g_k(z)\le_{cw} g_{k+1}(z)\le_{cw}g(z)$ for $k=0,1,\ldots$,
and for any $\zeta$ such that $|\zeta|=1$, $g(\zeta)$ solves the
equation in $\lambda$
\begin{equation}\label{eq:g}
a_1(z)\lambda^2 +(a_0(z)-1)\lambda+a_{-1}(z)=0,
\end{equation}
for $z=\zeta$, and $|g(\zeta)|\le 1$. 
Moreover, $g(1)=1$ if and only if $a_{-1}(1)\ge
a_1(1)$; if $a_{-1}(1)<
a_1(1)$, then $g(1)=a_{-1}(1)/a_1(1)$.
\end{theorem}
\begin{proof}
  Let us prove by induction on $k$ that
  $0\le_{cw}g_k(z)\le_{cw}g_{k+1}(z)$ and that $g_k(1)\le
  g_{k+1}(1)\le 1$.  For $k=0$ we have $g_0(z)=0$ and
  $g_1(z)=a_{-1}(z)$ so that $0\le_{cw} g_0(z)\le_{cw} g_1(z)$,
  moreover $g_0(1)=0\le g_1(1)=a_{-1}(1)\le 1$.  For the inductive
  step, assume $0\le_{cw} g_{k-1}(z)\le_{cw} g_{k}(z)$, $g_{k-1}(1)\le
  g_k(1)\le 1$ and prove that $0\le_{cw} g_{k}(z)\le_{cw}g_{k+1}(z)$
  and $g_k(1)\le g_{k+1}(1)\le 1$.  Since $a_i(z)\ge_{cw}0$, by the
  inductive assumption we have
  $g_{k+1}(z)=a_{-1}(z)+a_0(z)g_k(z)+a_1(z)g_k(z)^2\ge_{cw}
  a_{-1}(z)+a_0(z)g_{k-1}(z)+a_1(z)g_{k-1}(z)^2=g_k(z)\ge_{cw}0$ and
  $g_{k+1}(1)= a_{-1}(1)+a_0(1)g_k(1)+a_1(1)g_k(1)^2\le
  a_{-1}(1)+a_0(1)+a_1(1)=1$, moreover $g_{k+1}(1)=
  a_{-1}(1)+a_0(1)g_k(1)+a_1(1)g_k(1)^2\ge
  a_{-1}(1)+a_0(1)g_{k-1}(1)+a_1(1)g_{k-1}(1)^2=g_k(1)$.  Now we prove
  that the sequence $\{g_k(z)\}_k$ is a Cauchy sequence in the norm
  $\norm{\cdot}_\w$. For $k>h$, since $g_k(z)\ge_{cw}g_h(z)\ge_{cw}0$
  we have
\begin{equation}\label{eq:tmp1}
  \norm{g_k-g_h}_\w=g_k(1)-g_h(1).
\end{equation}
Since the sequence $\{g_k(1)\}_k$ is nondecreasing and bounded from above,
then it converges, thus it is a Cauchy sequence so that, in view of
\eqref{eq:tmp1} also $\{g_k(z)\}_k$ is a Cauchy sequence in the norm
$\norm{\cdot}_\w$. Since $\mathcal W$ is a Banach algebra, then
$\{g_k(z)\}_k$ converges in norm to $g(z)\in\mathcal W$ and $g(1)\le 1$.
Finally, for any given $\zeta$ such that $|\zeta|=1$, we have
$g(\zeta)=\lim_k g_k(\zeta)$ so that, by a continuity argument and in
view of \eqref{eq:gk}, $g(\zeta)$ solves equation
\eqref{eq:g}. Moreover, since $g(z)\ge_{cw}0$,
 then $|g(z)|\le g(1)\le
1$ for $|z|=1$. If $\zeta=1$, the solutions of equation \eqref{eq:g}
are $1$ and $a_{-1}(1)/a_1(1)$ (if $a_1(1)\ne 0$). Since $g(1)\le 1$, then $g(1)=1$ 
if and only if $a_{-1}(1)\ge
a_1(1)$. Moreover, if $a_{-1}(1)<
a_1(1)$, then $g(1)=a_{-1}(1)/a_1(1)$.
\end{proof}

We prove that for any $\zeta$ of modulus 1, the value $g(\zeta)$ is
the solution of minimum modulus of the equation \eqref{eq:g} where
$g(z)$ is the function of Theorem \ref{th:g}. This  can be shown by
using the following result and Lemma \ref{cor:1}, which weaken the
assumptions of \cite[Theorem 5.1]{bmm}.

\begin{lemma}\label{th:2} Assume that there exists $i\in\{-1,0,1\}$
  such that $|a_i(z)|<a_i(1)$ for any $z\ne 1$ with $|z|=1$. Then for
  any $\zeta\ne 1$ with $|\zeta|=1$, equation \eqref{eq:g} has a
  solution of modulus less than 1 and a solution of modulus greater
  than 1.
\end{lemma}
\begin{proof}
Let us prove that for any $\zeta\ne 1$ such that $|\zeta|=1$ there are
no solutions $\lambda$ of \eqref{eq:g} of modulus 1. By contradiction,
if $|\lambda|=1$ 
 then
$1=|\lambda|=|a_{-1}(\zeta)+a_0(\zeta)\lambda+a_1(\zeta) \lambda^2|\le
|a_{-1}(\zeta)|+|a_0(\zeta)|+|a_1(\zeta)|<|a_{-1}(1)|+|a_0(1)|+|a_1(1)|=1$
which is a contradiction. Now, define $f(x)=x(1-a_0(\zeta))$ and
$g(x)=x^2a_1(\zeta)+a_{-1}(\zeta)$ and observe that for $|x|=1$
 \begin{align*}
 |f(x)|&=|1-a_0(\zeta)|\ge 1-|a_0(\zeta)|\ge 1-a_0(1)=a_{-1}(1)+a_1(1),\\
 |g(x)|&\le |a_1(\zeta)|+|a_{-1}(\zeta)|\le a_1(1)+a_{-1}(1).
 \end{align*}
 Therefore, 
  $|f(x)|\ge |g(x)|$, moreover, the inequality is strict in
 view of the assumption $|a_i(z)|<a_i(1)$ for at least an index $i$.
 By applying Rouch\'e theorem \cite[Theorem 4.10b]{henrici:book} , it
 follows that $f(x)$ and
 $f(x)+g(x)=x^2a_1(\zeta)+(a_0(\zeta)-1)x+a_{-1}(\zeta)$ have the same
 number of roots in the open unit circle. On the other hand the
 function $f(x)$ has the only root $x=0$ since $1-a_0(\zeta)\ne 0$ for
 any $\zeta\ne 1$, $|\zeta|=1$.
\end{proof}

\begin{remark}\label{rem:alpha}
Observe that the condition $|a_i(z)|<a_i(1)$ can be equivalently
rewritten as $a_{i,j}= 0$ for at most one value of $j$ so that the
cases not covered by the above theorem are the ones where
$a_i(z)=\alpha_i z^{k_i}$ 
for $k_i\in\{-1,0,1\}$ and
$\alpha_{-1},\alpha_0,\alpha_1\ge 0$,
$\alpha_{-1}+\alpha_0+\alpha_1=1$. For instance, if $\alpha_i=1/3$ and
$k_i=i$, $i=-1,0,1$, then the quadratic equation has the double
solution $\lambda=\zeta^{-1}$ of modulus 1.
\end{remark}
  
The following result characterizes the case where equation
\eqref{eq:g} has two solutions with the same modulus.

\begin{lemma}\label{cor:1}
  Assume that $a_{i,j}\ge 0$, $\sum_{i,j=-1}^1 a_{i,j}=1$, and
  $a_1(z)\not\equiv 0$.  If for a given $\zeta$, $|\zeta|=1$, equation
  \eqref{eq:g} has two solutions $\lambda_1$, $\lambda_2$ such that
  $|\lambda_1|=|\lambda_2|$, 
  then there exists $k\in\{-1,0,1\}$ such
  that $\lambda_1=\lambda_2=\zeta^k$.
\end{lemma}
\begin{proof}
We use a continuity argument. Since $a_1(z)\not\equiv 0$, we assume
for simplicity that $a_{1,1}\ne 0$. Choose $0<\epsilon< a_{1,1}$
replace $a_{1,1}$ with $a_{1,1}-\epsilon$ and replace $a_{1,-1}$ with
$a_{1,-1}+\epsilon$. The new values of $a_{i,j}$ satisfy the
assumption of Lemma \ref{th:2}. Therefore, 
 there exist two solutions
$\lambda_1(\epsilon)$, $\lambda_2(\epsilon)$ such that
$|\lambda_1(\epsilon)|<1<|\lambda_2(\epsilon)|$ By letting
$\epsilon\to 0$ and setting $\lambda_i:=\lim_{\epsilon\to
  0}\lambda_i(\epsilon)$, then by continuity $|\lambda_1|\le 1\le
|\lambda_2|$, so that
$\lambda_1$
is still a, possibly non-unique, solution of minimum modulus
of \eqref{eq:g}. On the other hand if $|\lambda_1|=|\lambda_2|$, 
then necessarily $|\lambda_1|=|\lambda_2|=1$. If the assumption
of Lemma \ref{th:2} are satisfied, 
 then $\zeta=1$ and
$\lambda_1=1=\lambda_2$. If not, in view of Remark \ref{rem:alpha},
there exist $\alpha_i\ge 0$, $k_i\in\{-1,0,1\}$, $i=-1,0,1$, such that
$\alpha_{-1}+\alpha_0+\alpha_1=1$ and
$a_i(z)=\alpha_iz^{k_i}$, $i=-1,0,1$.
On the other hand since $|\lambda_1|=|\lambda_2|=1$, and
$\lambda_1\lambda_2=(\alpha_{-1}\zeta^{k_{-1}})/(\alpha_1 \zeta^{k_1})$,
then $\alpha_{-1}=\alpha_1$ so
that $\alpha_0=1-2\alpha_1$, $\alpha_1\le 1/2$. Thus,
$\lambda_1,\lambda_2$ solve the equation
$
\zeta^{k_1}\lambda^2+(\alpha_0\zeta^{k_0}-1)/\alpha_1\lambda+\zeta^{k_{-1}}=0$.
Since $|\lambda_1|=|\lambda_2|=1$,
 then
$|(\alpha_0\zeta^{k_0}-1)/\alpha_1|=|\lambda_1+\lambda_2|\le 2$, that
is, $|\alpha_0 \zeta^{k_0}-1|\le 1-\alpha_0$. Setting
$\zeta^{k_0}=\cos\theta +\cu \sin\theta$ the latter inequality turns
into $\alpha_0\le\alpha_0\cos\theta$.  This is possible if and only if
$\theta =0$ or $\alpha_0=0$.  In the former case we have either
$\zeta=1$ or $k_0=0$. If $\zeta=1$, 
 then $\lambda_1=\lambda_2=1$. If
$k_0=0$, 
 then $\lambda_1$ and $\lambda_2$ solve the equation $
\zeta^{k_1}\lambda^2+(\alpha_0-1)/\alpha_1\lambda +\zeta^{k_{-1}}=0$
that is $ \zeta^{k_1}\lambda^2-2\lambda +\zeta^{k_{-1}}=0$. The sum of
the solutions is $\lambda_1+\lambda_2=2/\zeta^{k_1}$ so that
$|\lambda_1+\lambda_2|=2$. Thus, necessarily we have
$\lambda_1=\lambda_2=\zeta^{-k_1}$.  In the remaining case
$\alpha_0=0$, we deduce that $\alpha_1=\alpha_{-1}=1/2$, so that the
quadratic equation is $\zeta^{k_1}\lambda^2 -2\lambda
+\zeta^{k_{-1}}=0$ and the same analysis applies.
\end{proof}

We may conclude with the following

\begin{theorem}\label{th:3.5}
  If $a_{i,j}\ge 0$ for $i,j = -1,0,1$, and $\sum_{i,j=-1}^1 a_{i,j}=1$, then
  for any $\zeta$ such that $|\zeta|=1$, the value $\theta=\lim_k
  g_k(\zeta)$ is the solution of minimum modulus of \eqref{eq:g}.
  Moreover, $\theta=g(\zeta)$ where $g$ is the function defined in
  Theorem \ref{th:g}.
\end{theorem}
\begin{proof}
  In the case where $a_1(z)\equiv 0$ the equation has only one
  solution which is the one of minimum modulus.  If $a_1(z)\not\equiv
  0$ Lemma \ref{cor:1} guarantees the existence of the minimal
  solution of \eqref{eq:g}. The claim follows from Theorem
  \ref{th:g}. Since $g_k(z)$ converges in the Wiener norm to $g(z)$,
  then $\lim_k g_k(\zeta)=g(\zeta)$.
\end{proof}

We will refer to the function $g(z)$ as to the minimal solution of
\eqref{eq:g}.

\subsection{Conditions for the compactness of $E_g$} 
In view of the  results of the previous section, under the only assumption
 $a_{i,j}\ge 0$ for $i,j = -1,0,1$ and  $\sum_{i,j=-1}^1 a_{i,j}=1$,
we may write 
\begin{equation}\label{eq:T+C}
G=T(g)+E_g
\end{equation}
 where $E_g:=G-T(g)$, and
 $\norm{E_g}_\infty\le\norm{G}_\infty+\norm{T(g)}_\infty\le 1+g(1)\le
 2$ so that $E_g\in \Endo{\ell^\infty}$, moreover we have $|E_g|e\le
 Ge+T(g)e\le 2e$. If $G\in \Endo{\ell^p}$, 
 then $\norm{E_g}_p\le
 \norm{T(g)}_p+\norm{G}_p\le\norm{g}_\w+\norm{G}_p$. We may synthesize
 this property in the following.

\begin{theorem}
The minimal nonnegative solution $G$ of the matrix equation in
\eqref{eq:mateqG1} can be written as $G=T(g)+E_g$ where
$g(z)\in\mathcal W$ is such that $g(1)\le 1$ and $g(z)$ is the
solution of minimum modulus of equation \eqref{eq:g}. Moreover,
$E_g\in \Endo{\ell^\infty}$ is such that $\norm{E_g}_\infty\le 1+g(1)$ and
$|E_g|e\le 2e$. Finally, if $G\in \Endo{\ell^p}$, 
 then $E_g\in \Endo{\ell^p}$.
\end{theorem}

Now, we are ready to  provide conditions under which $G$ belongs to
$\QTD$ or to $\mathcal{EQT}$. 

In \cite{lr:book} it is proven that the sequence $G_k$ generated by
\eqref{eq:3it} converges monotonically and point-wise to $G$. In
general, monotonic point-wise convergence does not imply convergence
in norm, as shown in the following example.  Let
$v^{(k)}:=(v^{(k)}_i)_{i \in \mathbb Z^+}$, where $v^{(k)}_i=\frac 1{(k+1)^i}$ for $k\ge
1$. It holds $v^{(k)}\in\ell^1$, $\lim_kv^{(k)}_i=0$ monotonically
but $\norm{v^{(k)}}_1=\frac k{k-1}$ so 
$\lim_k\norm{v^{(k)}}_1=1$. The example can be adjusted to the $p$
norm and extended to the case of matrices. In fact, the sequence
$A_k=v^{(k)}e^T$ is a sequence of compact operators in $\Endo{\ell^1}$
such that $\lim_k A_k=0$ where convergence is point-wise and
monotonic, but $\lim_k\norm{A_k}_1=1$.

Under the assumption $A_{-1}e>A_1e$, it is shown in
\cite[Theorem 4.2]{bmmj} that the sequence $\{G_k\}_k$ generated by \eqref{eq:3it}
converges in the infinity norm to $G$.  The following result slightly
weakens the assumptions and is the basis to prove that in this case
$G\in\QTD$.

\begin{theorem}\label{th:G}
If $A_{-1}e>A_1e$, or if $A_{-1}e\ge A_1e>0$, then for the sequence
$G_k$ generated by \eqref{eq:3it} we have
$\lim_k\norm{G_k-G}_\infty=0$.
\end{theorem}
\begin{proof}
Subtracting the equation $G_{k+1}=A_{-1}+A_0G_k+A_1G_k^2$ from the
equation $G=A_{-1}+A_0G+A_1G^2$ and setting $\mathcal E_k=G-G_k$, we
get
$
\mathcal E_{k+1}=A_0\mathcal E_k+A_1(\mathcal E_kG+G_k\mathcal E_k)$.
By proceeding similarly to the proof of Theorem 4.2 of \cite{bmmj}, we
may show that $\mathcal E_k\ge 0$, so that $\norm{\mathcal E_k}_\infty
=\norm{v_k}_\infty$ where $v_k=\mathcal E_ke$. Thus,
$
v_{k+1}=A_0v_k+A_1(\mathcal E_kGe+G_kv_k)\le (A_0+A_1+A_1G_k)v_k,
$
where we have used the property $Ge\le e$.  Whence we get
$\norm{v_{k+1}}_\infty\le \norm{v_k}_\infty\gamma_k$, for
$\gamma_k=\norm{A_0+A_{1}+A_1G_k}_\infty$.  On the other hand, since
$0\le (A_0+A_1+A_1G_k)e=(I-(A_{-1}-A_1G_k))e$, where we used the
identity $e=(A_{-1}+A_0+A_1)e$, and since
$\norm{A_0+A_1+A_1G_k}_\infty=\norm{(A_0+A_1+A_1G_k)e}_\infty$, we have
$\gamma_k=\norm{(I-(A_{-1}-A_1G_k))e}_\infty$. Therefore, 
 $\gamma_k<1$ if
and only if the vector $w_k:=(A_{-1}-A_1G_k)e$ has positive components
which do not decay to zero. Since $G_k$ has finite support, the vector
$A_1G_ke$ has finite support so that the condition $a_{-1}(1)\ne 0$
implies that the components of $w_k$ do not decay to zero. Thus, it is
enough to prove that $w_k>0$.  Since $G\ge G_k$, then $Ge\ge G_ke$ so
that $(A_{-1}-A_1G_k)e\ge (A_{-1}-A_1)e$. Whence the condition
$(A_{-1}-A_1)e>0$ implies that the vector $w_k$ has positive
components.  In the case where $(A_{-1}-A_1)e\ge 0$ and $A_1e>0$, we
may prove by induction that $G_ke<e$. In fact, for $k=0$ the property
holds since $G_0=0$.  For the implication $k\to k+1$ we have
$
G_{k+1}e=(A_{-1}+A_0G_k+A_1G_k^2)e\le (A_{-1}+A_0+A_1G_k)e<(A_{-1}+A_0+A_1)e=e, 
$
where we used the fact that $A_1G_ke<A_1e$ since $G_ke<e$ and $A_1$
has at least a nonzero entry in each row since by assumption $A_1e>0$.
From the property $G_ke<e$ we get $A_1G_ke<A_1e$ so that $
w_k=(A_{-1}-A_1G_k)e>(A_{-1}-A_1)e\ge 0$.
\end{proof}

\begin{remark}\label{rem:1}\rm Recall that the condition $A_{-1}e>A_1e$
  implies that $a_{-1}(1)>a_1(1)$ while 
the condition $A_{-1}e\ge A_1e$ implies that $a_{-1}(1)\ge a_1(1)$. In
both cases the quadratic equation
$a_1(1)\lambda^2+(a_0(1)-1)\lambda+a_{-1}(1)=0$ has two real solutions
$\lambda_1=1$ and $\lambda_2=a_{-1}(1)/a_{1}(1)$.  Moreover
$\lambda_1=1$ is the minimal solution. In particular, in view of
Theorem \ref{th:g}, we have $g(1)=1$.  Conversely, if $g(1)=1$ is
the minimal solution of the above quadratic equation,
 then, for Theorem~\ref{th:g}, $a_{-1}(1)\ge a_1(1)$.
\end{remark}

The convergence properties of the sequence $\{G_k\}_k$ stated by Theorem
\ref{th:G} allow to provide sufficient conditions under which
$G\in\QTD$.

\begin{theorem}\label{th:4.11}
If  $\lim_k\|G_k-G\|_\infty=0$, 
then the minimal
nonnegative solution $G$ of the matrix equation \eqref{eq:mateqG1} can
be written as $G=T(g)+E_g$ where $g(z)\in\mathcal W$ is the minimal
solution of \eqref{eq:g}, 
and $E_g\in \Endo{\ell^\infty}$
has the decay property. 
\end{theorem}
\begin{proof}
Consider the sequence $G_k=T(g_k)+E_k\in\QTD$ generated by
\eqref{eq:3it}, where $g_k(z)\in\mathcal W$ and $E_k$ has finite
support.  Concerning the first part, we observe that
$\norm{E_k-E_g}_\infty\le\norm{G_k-G}_\infty+\norm{T(g_k)-T(g)}_\infty$.
Thus, since
$\norm{T(g_k)-T(g)}_\infty=\norm{T(g-g_k)}_\infty=\norm{g-g_k}_\w$, in
view of Theorem \ref{th:g} we have
$\lim_{k}\norm{T(g_k)-T(g)}_\infty=0$. Since $\lim_k\norm{G_k-G}_\infty=0$, we conclude that
$\lim_k \norm{E_k-E_g}_\infty=0$. Since $E_k$ has finite support, 
 then
it has the decay property so that, for Theorem \ref{thm:decayclosure},
$E_g$ has the decay property as well.  

\end{proof}

From Theorem \ref{th:G} the condition $A_{-1}e>A_1e$, which is equivalent to
$a_{-1}(1)>a_1(1)$ and $y_{-1}(1)>y_1(1)$, implies $\lim_k \| G_k-G\|_\infty=0$. We will weaken the
assumptions of  Theorem \ref{th:G} by removing the boundary condition
$y_{-1}(1)>y_1(1)$.  To this aim, consider the correction
$E_g=G-T(g)\in \Endo{\ell^\infty}$, where $G$ is the minimal nonnegative
solution to the equation \eqref{eq:mateqG1} and $g(z)$ is the solution
of minimum modulus to the equation \eqref{eq:g} which exists under the
assumptions of Theorem \ref{th:3.5}. Observe that if $E_g$ has not the decay
property, 
 then $w=|E_g|e$ is such that $\norm{w}_\infty<\infty$ but
$\lim_i w_i$, if it exists, is not zero. 

The following lemma is needed to prove the main result of this
section.  The only assumption needed is that
$a_1(1)+a_{-1}(1)>0$. This condition is very mild since it excludes
only the case where $a_{i,j}=0$ for $i=1,-1$ and for any $j$.

\begin{lemma}\label{lem:2}
Assume that $a_1(1)+a_{-1}(1)>0$ and define
$
\psi(z)=\frac{a_1(z)}{1-a_0(z)-a_1(z)g(z)},~ \hbox{for }|z|=1$.
 Then $\psi(z)\in\mathcal W$, 
$\psi(z)\ge_{cw} 0$, $\norm{\psi}\w=\psi(1)$ and for $G=T(g)+E_g$
we have
\begin{equation}\label{eq:ult}
E_g\doteq T(\psi^k)E_gG^k, \quad k=0,1,2,\ldots .
\end{equation}
\end{lemma}
\begin{proof}
We show that the function $\varphi(z)=1-\gamma(z)$,
$\gamma(z)=a_0(z)+a_1(z)g(z)$, is such that $\varphi(z)\ne 0$ for
$|z|=1$. Since $\gamma(z)\ge_{cw} 0$, then $|\gamma(z)|\le \gamma(1)$,
so that it is sufficient to prove that $\gamma(1)<1$.  We have
$
\gamma(1)=a_0(1)+a_1(1)g(1)\le a_0(1)+a_1(1)=1-a_{-1}(1)$.
Therefore, if $a_{-1}(1)>0$, 
 then $\gamma(1)<1$. On the other hand, if
$a_{-1}(1)=0$, 
 then $g(1)=0$ and $a_{1}(1)>0$ since, by assumption,
$a_1(1)+a_{-1}(1)>0$, so that
$
\gamma(1)=a_0(1)=1-a_{1}(1)<1$.
This way, $\psi(z)=a_1(z)/\varphi(z)\in\mathcal W$. Moreover, since 
$\sum_{k=0}^\infty \gamma(1)^k=1/(1-\gamma(1))<\infty$, and $\gamma(z)\ge_{cw}0$, then 
 $\sum_{k=0}^\infty \gamma(z)^k\in\mathcal W$ and coincides
with $1/\varphi(z)$. Moreover, since $\gamma(z)\ge_{cw}0$, 
then $1/\varphi(z)\ge_{cw}0$ and $\psi(z)\ge_{cw}0$.
From the condition $A_1G^2+(A_0-I)G+A_{-1}=0$, relying on Lemma
\ref{lem:ideal} and Corollary \ref{cor:T(ab)}, we obtain 
\begin{equation}\label{eq:st}
T(a_1)E_gG\doteq T(1-a_0-a_1g)E_g.
\end{equation}
By multiplying to the left both sides of \eqref{eq:st} by
$T(1/\varphi(z))$, in view of \eqref{eq:dec}, we get
\[
T(\psi)E_gG\doteq E_g,\quad \psi(z)=\frac{a_1(z)}{1-a_0(z)-a_1(z)g(z)}.
\]
Finally, by multiplying the above equation to the left by $T(\psi)$
and to the right by $G$, by means of the induction argument, we get
\eqref{eq:ult}.
\end{proof}

It is interesting to point out that if $a_{-1}(1)\ne 0$, 
 then the
function $\psi(z)$ can be written in a simpler form as
$\psi(z)=g(z)\frac{a_1(z)}{a_{-1}(z)}$.

We are ready to prove the main theorem of this section which provides
conditions under which $G\in\QTD$ or
$G\in\mathcal{EQT}$.

\begin{theorem}\label{th:4.13}
Assume that $a_{-1}(1)+a_1(1)>0$.  Let $G$ be the minimal nonnegative
solution of \eqref{eq:mateqG1} decomposed as $G=T(g)+E_g$, where
$g(z)$ is the minimal solution of \eqref{eq:g} and $E_g:=G-T(g)$. Then the following properties hold:
\begin{enumerate}
\item If $a_{-1}(1)>a_1(1)$, then $E_g$ has the decay property.
\item If $a_{-1}(1)< a_1(1)$ and $\lim_k
\norm{G^k}_\infty=0$, 
 then $E_g$ has the decay property. 
\item 
If
$a_{-1}(1)< a_1(1)$, $G$ is stochastic and strongly ergodic,
that is $\lim_k\|G^k-e\pi_g^T\|_\infty=0$, and $\pi_g^TG=\pi_g^T$, $\pi_g^Te=1$,
then
$E_g=(1-g(1))e\pi_g^T+S_g$, where $S_g$ has the decay property.
\item  If $G$ is stochastic and $E_g$ has the decay property, then $a_{-1}(1)\ge
a_1(1)$ and $g(1)=1$.
\end{enumerate}
\end{theorem}
\begin{proof}
The proof of properties 1--3 relies on equation \eqref{eq:ult} and on the limit for
$k\to\infty$ of its right-hand side. This limit depends on the value of
$\|\psi\|_\w=\psi(1)$, where $\psi(z)$ is defined in Lemma
\ref{lem:2}. Therefore, 
 we show that either $\psi(1)=1$ or $\psi(1)<1$
and we deduce the properties of $E_g$ accordingly.
 Observe that if $a_{-1}(z)=0$, 
 then
$a_0(1)+a_1(1)=1$ and $g(1)=0$ so that $\psi(1)=1$.  If $a_{-1}(z)\ne
0$, for Theorem \ref{th:g} we may distinguish two cases: the case
where $a_{-1}(1)/a_1(1)> 1$ and the case $a_{-1}(1)/a_1(1)< 1$. In
the first case $g(1)=1$ so that $\psi(1)=a_1(1)/a_{-1}(1)<1$. In the
second case $g(1)=a_{-1}(1)/a_1(1)$ so that $\psi(1)=1$.
Consider the case $a_{-1}(1)>a_1(1)$. Since $g(1)=1$, 
 then
$\psi(1)=a_1(1)/a_{-1}(1)<1$. Moreover, since $\psi(z)\ge_{cw}0$, 
 then
$\norm{\psi^k}_\w=\psi(1)^k$, whence $\lim_k
\norm{\psi^k}_\w=\lim_k\psi(1)^k=0$. Therefore, 
$\lim_k\norm{T(\psi^k)}_\infty=0$.  On the other hand, since $Ge\le e$
and $G\ge 0$, 
 then $\norm{G^k}_\infty\le 1$. Whence, since $E_g\in
\Endo{\ell^\infty}$, 
 then from equation \eqref{eq:ult} in Lemma
\ref{lem:2} we have $\lim_k\norm{T(\psi^k)E_g G^k}_\infty\le
\lim_k\norm{T(\psi^k)}_\infty\norm{E_g}_\infty\norm{G^k}_\infty=0$.
That is, the sequence $\{F_k\}_k$, $F_k=E_g-T(\psi^k)E_gG^k$, is such
that $F_k\doteq 0$ and $\lim_k\|E_g-F_k\|_\infty=0$. In view of
Theorem~\ref{thm:decayclosure}, applied to the sequence $\{F_k\}_k$, we
conclude that $E_g\doteq 0$ so that $E_g$ fulfills the decay property.
Now, consider the case $a_{-1}(1)< a_1(1)$. Observe that since
$\psi(1)=1$, 
 then $\norm{\psi}_\w=\psi(1)=1$ and
$\norm{\psi^k}_\w=\psi(1)^k=1$, therefore, 
$\norm{T(\psi^k)}_\infty=\psi(1)^k=1$. If $\lim_k \norm{G^k}_\infty=0$
then, 
 taking the limit in \eqref{eq:ult}, in view of
Theorem~\ref{thm:decayclosure} applied to the sequence $\{F_k\}_k$ we
deduce that $E_g$ has the decay property. On the other hand, if the
Markov chain associated with the matrix $G$ is strongly ergodic, that
is, $\lim_k\norm{G^k-e\pi_g^T}_\infty=0$, we have $G^k=e\pi_g^T+R_k$
where $\lim_k\norm{R_k}_\infty=0$. Therefore,
\[
E_g-T(\psi^k)E_ge\pi_g^T\doteq \widehat E_k,\quad \widehat E_k=T(\psi^k)E_gR_k.
\]
Since $\norm{\widehat E_k}_\infty\le 
\norm{T(\psi^k)}_\infty\norm{E_g}_\infty\norm{R_k}_\infty=
\norm{E_g}_\infty\norm{R_k}_\infty$, then $\lim_k\norm{\widehat E_k}_\infty=0$.
Now, define $A=E_g-(1-g(1))e\pi_g^T$ and $A_k=A-\widehat E_k$. 
Since $E_ge=Ge-T(g)e\doteq (1-g(1))e$ and $T(\psi^k)e\doteq e$, 
 then 
$(1-g(1))e\doteq T(\psi^k)E_ge$ whence
\[\begin{aligned}
A_k&= E_g-(1-g(1))e\pi_g^T-\widehat E_k\doteq 
E_g-T(\psi^k)E_ge\pi_g^T-T(\psi^k)E_gR_k\\
&=E_g-T(\psi^k)E_g(e\pi_g^T+R_k)
=E_g-T(\psi^k)E_gG^k\doteq 0
\end{aligned}
\]
in view of \eqref{eq:ult}, thus $A_k\doteq 0$. Since
$\lim_k\norm{A-A_k}_\infty=0$ we may apply
Theorem~\ref{thm:decayclosure} and conclude that $A\doteq 0$, that is
$E_g\doteq (1-g(1))e\pi_g^T$, in other words $E_g=(1-g(1))e\pi_g^T+S_g$ where $S_g$ has the decay property.
Concerning the last property, consider
$w:=|E_g|e=|G-T(g)|e\ge |Ge-T(g)e|=|e-T(g)e|$. Since by assumption,
$\lim_i w_i=0$, 
 then $\lim_i(T(g)e)_i=1$. On the other hand, since
$g(z)\in\mathcal W$ has nonnegative coefficients, 
 then
$\lim_i(T(g)e)_i=g(1)$, so that $g(1)=1$.  Since $g(1)=1$ is the
minimal nonnegative solution of the scalar equation
$a_1(1)\lambda^2+(a_0(1)-1)\lambda+a_{-1}(1)=0$, in view of Theorem \ref{th:g}, it follows that $a_{-1}(1)\ge a_1(1)$.
\end{proof}


\section{Applications and numerical results}
\label{sec:numerical}
This section is devoted to validate the computational framework on some
applications of 1-D and 2-D random walks,  
 which require the extended algebras  $\QTD$ and $\EQT$. The
experiments are carried out on a PC with 
a Xeon E5-2650 CPU running at 2.20 GHz, restricted to $8$ cores
and $10$ GB of RAM. 
The implementation relies on
the \cqttoolbox~\cite{bmr}, and the package {\tt SMCSolver} of \cite{smcsolver}, tested under MATLAB2019a. We have used
the tolerance $10^{-14}$ for truncation and compression in the \cqttoolbox.

\subsection{1D random walk with reset}\label{sec:rw}

Here, we consider a discrete time
Markov chain on the set of states $\mathbb N$, whose probabilities 
 of
left/right jumps 
 are independent of the current state,
with the only exception of the boundary condition.
In this setting the
transition probability matrix $P$  takes
the form
\[
P=
\begin{bmatrix}
b_0&a_1& a_2& a_3 & \ldots\\
b_{-1} & a_{0} &a_1 & a_2 & \ddots\\
b_{-2} & a_{-1} &a_{0}&a_1& \ddots \\
\vdots& \vdots &\ddots&\ddots&\ddots
\end{bmatrix},
\]
where the entries are nonnegative and such that $b_{-i}=1-\sum_{j=-i+1}^\infty a_j$, for $i=0,1,2,\ldots$.
Observe that, if $\sum_{j\in\mathbb{Z}}a_j=\gamma<1$, then
$\lim_{i\to\infty}b_{-i}=1-\gamma$, hence $P\in\EQT$.

	\begin{figure}\label{fig:1}
	\begin{center}
		\begin{tabular}{cc}
		\resizebox{6cm}{!}{
	\begin{tikzpicture}
	\begin{loglogaxis}[
		legend pos = north west, width = .7\linewidth, height = .38\textheight,
		xlabel = {Maximum skip length $m$}, 
		ylabel = {CPU time}, title = {$\gamma = 0.90$}]
	    \addplot table {example_1d_threshold1e14_0.90.dat}; 
	    \addplot table[x index = 0, y index = 2] {example_1d_threshold1e14_0.90.dat};
	    \addplot table[x index = 0, y index = 3] {example_1d_threshold1e14_0.90.dat};
	    \legend{QT linear system, Squaring, MAM + CR}
	\end{loglogaxis}
	\end{tikzpicture}}
&	\resizebox{6cm}{!}{
\begin{tikzpicture}
		\begin{loglogaxis}[
		legend pos = north west, width = .7\linewidth, height = .38\textheight,
		xlabel = {Maximum skip length $m$},
		ylabel = {CPU time}, title = {$\gamma = 0.99$}]
	    \addplot table {example_1d_threshold1e14_0.99.dat}; 
	    \addplot table[x index = 0, y index = 2] {example_1d_threshold1e14_0.99.dat};
	    \addplot table[x index = 0, y index = 3] {example_1d_threshold1e14_0.99.dat};
	    \legend{QT linear system, Squaring, MAM + CR}
	\end{loglogaxis}
	\end{tikzpicture}}
\\
\resizebox{6cm}{!}{
	\begin{tikzpicture}
	\begin{loglogaxis}[
		legend pos = north west, width = .7\linewidth, height = .38\textheight,
		xlabel = {Maximum skip length $m$}, 
		ylabel = {Residual error $\norm{\pi^T - \pi^T P}_1$}, 
		title = {$\gamma = 0.90$}, ymax=5e-12]
	    \addplot table[x index = 0, y index = 4] {example_1d_threshold1e14_0.90.dat};
	    \addplot table[x index = 0, y index = 5] {example_1d_threshold1e14_0.90.dat};
	    \addplot table[x index = 0, y index = 6] {example_1d_threshold1e14_0.90.dat};
	    \legend{QT linear system, Squaring, MAM + CR}
	\end{loglogaxis}
	\end{tikzpicture}}
	&	\resizebox{6cm}{!}{
	\begin{tikzpicture}
	\begin{loglogaxis}[
		legend pos = north west, width = .7\linewidth, height = .38\textheight,
		xlabel = {Maximum skip length $m$}, 
		ylabel = {Residual error $\norm{\pi^T - \pi^T P}_1$}, 
		title = {$\gamma = 0.99$}, ymax=5e-12] 
	    \addplot table[x index = 0, y index = 4] {example_1d_threshold1e14_0.99.dat};
	    \addplot table[x index = 0, y index = 5] {example_1d_threshold1e14_0.99.dat};
	    \addplot table[x index = 0, y index = 6] {example_1d_threshold1e14_0.99.dat};
	    \legend{QT linear system, Squaring, MAM + CR}
	\end{loglogaxis}
	\end{tikzpicture}}	
\end{tabular}	\caption{1-D Random walk with maximum skip length $m$ and reset with probability $1-\gamma$. CPU time in seconds (top line) and residual errors (bottom line) in the computation of the vector $\pi$ for two values of $\gamma$ 
and for three different algorithms: Solving a QT linear system, performing repeated squarings, applying Cyclic Reduction.}
\end{center}
\end{figure}
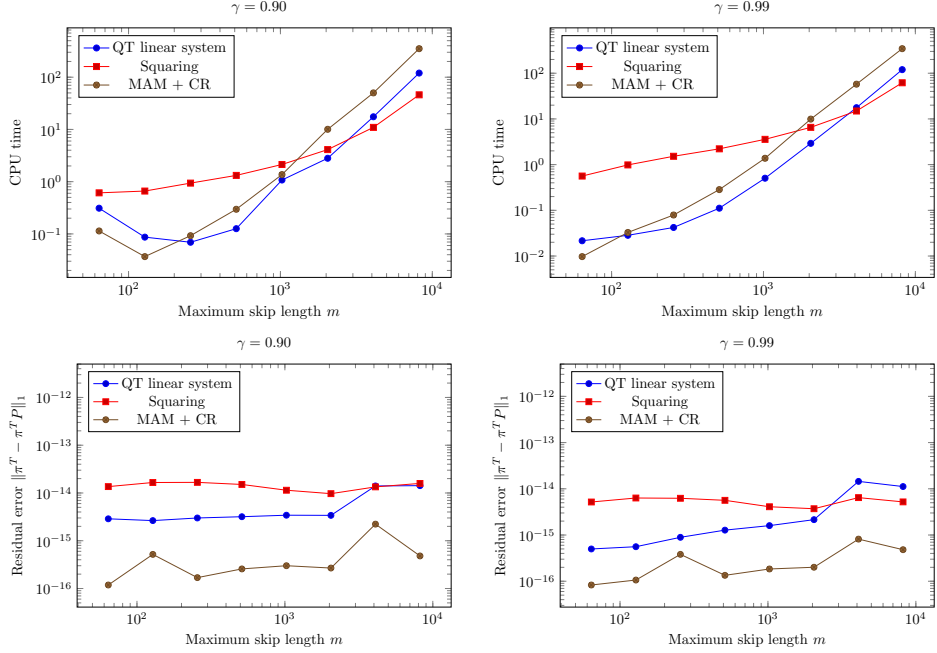

Recently some interest has been raised by models that incorporate
exogenous drastic events. Examples might include catastrophes,
rebooting of a computer or a strike causing a shutdown in the
transportation system. This is modeled by  a
random walk on $\mathbb N$ whose transitions allow to reach an
initial state from every state. 
Indeed, if
$a_j=0$ for $j<-m$, where $m\ge 1$, and if $\sum_{j=-m}^\infty a_j=\gamma<1$, then from any state
$k\ge m $ the process can reach state 0 with probability
$1-\gamma$. In other words, when the process is in any state $k\ge m
$, it is reset with probability $1-\gamma$.

The transition matrix $P$ generalizes the well studied Markov processes of M/G/1 and G/M/1-type, having an upper and lower Hessenberg structure, respectively \cite{bini2005numerical},  \cite{neuts:book}. These Markov processes are used to model a wide variety of queueing problems  \cite{alfa:book}, \cite{he:book}.
In particular, the case of models with reset has been analyzed in  \cite{hg:mam}, \cite{vhb:mam}, \cite{vhb}  and  \cite{vvb}.
Assume that the matrix $P$ is irreducible. If $\gamma\ne 0$, then the Markov chain is positive recurrent \cite[Theorem 5.3]{bini2005numerical} so that there exists  the steady state vector $\pi$
such that $\pi^TP=\pi^T$, $\pi^Te=1$. If $a_j=0$ for $|j|>m$, where $m\ge 1$, the matrix $P$ can be partitioned into $m\times m$ dimensional blocks, thus obtaining a matrix of the form
\[
P=
\begin{bmatrix}
W_0&V_1& &  &{\Large 0}\\
W_{-1} & V_{0} &V_1 &  & \\
W_{-2} & V_{-1} &V_{0}& V_1 & \\
W_{-3} &          &V_{-1}&V_0&\ddots\\
\vdots&{\Large 0}&      &\ddots&\ddots 
\end{bmatrix}.
\]
The vector $\pi$, partitioned into $m$-dimensional vectors $\pi_i$, $i=0,1,\ldots$, can be computed by means of the recursion $\pi_{i+1}^T=\pi_i^T R$, $i=0,1,\ldots$ where $\pi_0$ solves the equation 
$\pi_0^T(I-\sum_{i=0}^\infty R^i W_{-i})=0$,
$\pi_0^T(I-R)^{-1}e=1$ and $R$ is the minimal nonnegative solution of the equation $X=X^2V_{-1}+XV_0+V_1$ (see \cite[Theorem 5.4]{bini2005numerical}, \cite{neuts:book}). This strategy for computing  $\pi$ is known  as Matrix Analytic Method \cite{neuts:book}.

In our case, we can decompose $P=T+e v^T$,
where $T\in\QTD$ is semi-infinite quasi-Toeplitz and
$v^T=(1-\gamma,0,\ldots)$, and get the relation
\[
 \pi^T= \pi^TP=\pi^TT+(\pi^Te)v^T=\pi^TT+v^T.
\]
This yields $ \pi^T(I-T)=v^T $ that enables to retrieve $\pi^T$ by
solving a linear system with the matrix $I-T$ in
$\QTD$. Note that, in this case, the class $\QTD$ is
used both in the formulation of the problem and in the algorithmic
procedure which is simply reduced to the application of the Matlab
backslash command available in the extended \cqttoolbox \cite{bmr}, see Section~\ref{sec:eqt-tool}.

A different algorithmic approach, which exploits the computational
properties of the class $\mathcal{EQT}$, is to apply the power method
implemented by means of the repeated squaring technique to generate
the sequence $P_{k+1}=P_k^2$, $k\ge 0$, starting with $P_0=P$, which converges
quadratically to the limit $e\pi^T$.
In this case, since $\mathcal{EQT}$ is an algebra, all the matrices $P_k$
belong to $\mathcal{EQT}$ and can be computed by means of the command
{\tt P = P*P;} available in the extended arithmetic of the
\cqttoolbox{}, see Section~\ref{sec:eqt-tool}.

 We assume the following configuration for the
transition probabilities: $a_j=\frac{\theta\sigma_j}{j^4}$, for $-m\le j\le m$, where $\sigma_j$ is a random number uniformly distributed in $[1,2]$,   $a_j=0$ for $|j|>m$, and $\theta$ is chosen in such a way that $\sum_j a_j=\gamma\in[0,1]$. The values $b_j$ are such that $P$ is stochastic. Except for the first column, the matrix $P$ is a Toeplitz matrix with bandwidth $2m+1$.
The experiments have been run $100$ times and the results 
for residuals and timings have been averaged. 

We have compared the two algorithms above and the Matrix Analytic Method (MAM) where we used the algorithm of cyclic reduction (CR) from the package {\tt SMCSolver} for solving the matrix equation. It is worth saying that CR is one of the fastest algorithms customarily used to solve this kind of problems for finite matrices.
Figure \ref{fig:1} reports CPU time and the residual error $\norm{\pi^T - \pi^T P}_1$ in computing the vector $\pi$ for two different values of $\gamma=0.9, 0.99$ and for $m$ taking values in the range $[2^6,2^{13}]$. We may observe that the algorithms based on our approach perform faster than the algorithm based on the combination of CR and the reblocking technique. 
For instance, for $m=2^{13}$ independently of the value of $\gamma$, the method based on  the combination of CR and the reblocking technique takes 350 seconds while the method based on the ``backslash'' command takes 120 seconds and the method based on repeated squarings takes just 46 seconds and 62 seconds for $\gamma=0.9$ and $\gamma=0.99$, respectively, that is, it is about 8 times faster. 
Concerning the accuracy, all the algorithms have a good performance, with 
the one based on CR performing slightly better. The approaches using 
\texttt{cqt-toolbox} achieve an accuracy within the magnitude of the 
chosen truncation threshold, which is set to $10^{-14}$.

\subsection{Two-node Jackson network with reset}\label{sec:tandem}
Here, we consider the Two-node Jackson network of \cite{MT} modified
by allowing a reset.  This model, represented by a continuous time
Markov chain, is described in Figure \ref{fig:MT} and consists of two
queues $Q_1$ and $Q_2$ with buffers of infinite capacity.  Customers
arrive at $Q_1$ and $Q_2$ according to two independent Poisson
processes with rates $\lambda_1$, $\lambda_2$.  Customers are served
at $Q_1$ and $Q_2$ with independent service times exponentially
distributed with rates $\mu_1$ and $\mu_2$, respectively.  On leaving
$Q_1$, two events may occur: either there is a reset of the queue
where all the customers waiting to be served in $Q_1$ leave the
system, this happens with probability $1-\gamma$ for $0<\gamma<1$; or,
with probability $\gamma$, one customer exits from $Q_1$. The latter enters
$Q_2$ with probability $p$ or leaves the system with probability $1 -
p$, where $0 < p < 1$.  After completing service at $Q_2$, the
customer may enter again $Q_1$ with probability $q$ or may leave the
system with probability $1 - q$, where $0 < q < 1$.

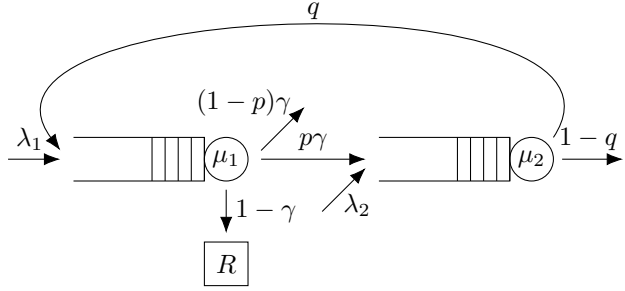
\begin{figure}
  \centering

  \begin{tikzpicture}[scale=0.58]
    \def\arrowscale{1.3}
    \def\myarr{-{Latex[scale=\arrowscale]}}
    \draw[\myarr] (-1.5,.5) -- (-0.3,.5);
    \node[anchor=south] at (-1.0,0.5) {$\lambda_1$};
    
    \draw (0,0) -- (3,0) -- (3,1) -- (0,1);
    \foreach \j in {0, 0.3, ..., 1.5} {
      \draw (3-\j,0) -- (3-\j,1);
    }

    \draw (3.5,0.5) circle (0.5);
    \node at (3.5,0.5) {$\mu_1$};

    \draw[\myarr] (4.3,.7) -- (5.3,1.7);
    \node at (3.9,1.8) {$(1-p)\gamma$};

    \draw[\myarr] (4.3,.5) -- (6.7,.5);
    \node[anchor=south] at (5.5,.5) {$p\gamma$};


    \draw[\myarr] (5.7,-.7) -- (6.7,.3);
    \node at (6.5,-.55) {$\lambda_2$};

    \draw (7,0) -- (10,0) -- (10,1) -- (7,1);
    \foreach \j in {0, 0.3, ..., 1.5} {
      \draw (10-\j,0) -- (10-\j,1);
    }

    \draw (10.5,0.5) circle (.5);
    \node at (10.5,0.5) {$\mu_2$};

    \draw[\myarr] (3.5,-0.2) -- (3.5,-1.2);
    \node[anchor=west] at (3.5,-0.7) {$1-\gamma$};
    \draw (3,-1.4) rectangle (4,-2.4);
    \node at (3.5,-1.9) {$R$};

    \draw[\myarr] (11.2,.5) -- (12.6,.5);
    \node[anchor=south] at (11.8,.5) {$1-q$};

    \draw[\myarr] (11,1.0) to [bend right=125] (-0.3,.7);
    \node at (5.5,3.9) {$q$};
  \end{tikzpicture}

  \caption{Pictorial description of the transitions for the 
    two node Jackson network with reset. The queue $Q_1$ is on the left, the queue $Q_2$ is on the right. The square denoted by $R$
    indicates the reset event which is triggered with probability
  $1-\gamma$ after service at the queue $Q_1$.}
  \label{fig:MT}
\end{figure}

The probability matrix, obtained after uniformization from the generator matrix
encoding the transition rates \cite{lr:book}, is given by $P=\qtoep(B_0,B_1;\,A_{-1},A_0,A_1)$ where 
\begin{equation}\label{eq:tandem}
\begin{aligned}
& A_{-1}=\frac 1\theta \qtoep((1-q)\mu_2,q\mu_2;\, 0, (1-q)\mu_2,q\mu_2),\\
& A_0=\frac 1\theta \qtoep(\gamma\mu_1, \lambda_1;\, \gamma(1-p)\mu_1,0,\lambda_1)+\frac{1-\gamma}\theta ee_1^T,\\
& A_1=\frac 1\theta \qtoep(\lambda_2,0;\, \gamma p\mu_1 , \lambda_2,0),\\
& B_0=A_0+\frac{\mu_2}\theta I,\quad
B_1=A_1,
\end{aligned}
\end{equation}
and $\theta=1-\gamma+\gamma\mu_1+\mu_2+\lambda_1+\lambda_2$.
In this example we have $A_1,A_{-1}\in\QTD$ and
$A_0\in\mathcal{EQT}$. In this case $G\in\EQT$, so that it can be written as $ G= T( g)+E_g+ e v^T$, where $g$ is the solution of \eqref{eq:g}, $E_g$ has the decay property and $v\in\ell^1$.

Several generalizations of this model are possible. For instance, 
we may allow different reset levels or we may allow reset also in the second queue $Q_2$. In that case we would obtain a GI/M/1 Markov chain with semi-infinite blocks as those analyzed in \cite{lato:varese}.

 The parameters are set as follows:
$\lambda_1=2, \mu_1=3, \lambda_2=1,\mu_2=2,p=0.3, q=0.2$, with two different values of $\gamma$, namely, $\gamma=0.95$ and $\gamma=0.99$.  The symbol $g$ is computed once for all by means of the evaluation-interpolation algorithm of \cite{bmmj}.
We solve
equation \eqref{eq:mateqG1}, with coefficients defined as in
\eqref{eq:tandem}, by means of the iteration 
$
X_{k+1}=(I-A_0-A_1X_k)^{-1}A_{-1},
$
analyzed in \cite{bmmj},
 with $X_0=T(g)+(I-T(g))ee_1^T\in\EQT$,  for different values of the required output accuracy, obtained by modifying the parameter {\tt threshold} in the \cqttoolbox. The residual errors of the approximated solutions obtained this way and the CPU times are computed.  

In certain cases it is possible to express explicitly the vector $\pi$ in product form. In view of the results in \cite{gos19}, in our case it is not possible to provide this explicit representation of $\pi$.

 We compared this approach (QT-based method) with a truncation based algorithm (truncation method). This method, inspired by  \cite{lt:trunc}, is based on a heuristic for recovering the solution $G$ by the finite dimensional solution $G_k$ of the equation obtained by truncating to a finite size $k$ 
the infinite coefficients $A_{-1},A_0,A_1$.
More specifically, we expect that the $(k/2)\times (k/2)$  leading principal submatrix $G_{k,1/2}$ of $G_k$ is a good approximation of the leading principal  $(k/2)\times (k/2)$ submatrix of $G$, for sufficiently large values of $k$.
Therefore, by defining
$T_m$ the $m\times m$ leading principal submatrix of $T(g)$, 
  for the decay properties of $E_g$,
the last row $v_{k}$ of $G_{k,1/2}-T_{k/2}$ provides an approximation of the first $k/2$ components of $v$. 
The matrix $G_{k,1/2}$ is written as $G_{k,1/2}=T_{k/2}+ev_k^T +C_k$, so that
$C_k=G_{k,1/2}-T_{k/2}-ev_k^T$.
The approximated solution $\widehat G$ is defined as $\widehat G=T( g)+\widehat E_g+e\hat v^T$ where $\widehat E_g$ is the infinite matrix obtained by filling with zeros the matrix $C_k$ and $\hat v$ is the infinite vector obtained by filling with zeros the vector $v_k$. 
The finite dimensional minimal nonnegative solution $G_k$ is computed by means of the function {\tt QBD\_CR} of {\tt SMCSolver} \cite{smcsolver}.  
The residual error of $\widehat G$ are plotted against the CPU time needed for its computation, for increasing values of $k$. It is not easy to determine
a priori the value of $k$ required to reach a certain accuracy, so we have 
chosen the values of $k$ a posteriori to attain residuals
in the interval $[10^{-12}, 10^{-2}]$. In the considered experiment, 
this means a maximum size of $k = 1000$ for $\gamma = 0.95$, and $k = 5000$ for $\gamma = 0.99$.

In Figure \ref{fig:2} we plot the pairs (CPU time, residual errors) for the two different approaches and for two different values of the reset probability $1-\gamma$. The residual errors are computed as $R(G) := \norm{A_{-1} + A_0 G + A_1G^2 - G}_\infty$. We may see that for values of  $\gamma$ close to 1, in order to reach an approximation error closer to the machine precision, the QT-based approach is much faster than the method obtained by truncating the matrix to finite size. In particular, for $\gamma=0.99$, 
the truncation method requires about $20$ minutes to get the same accuracy that the QT-based technique obtains in about $10$ seconds. On the other
hand, for $\gamma = 0.95$, the QT-based method is slightly slower, but overall
the two methods perform comparably. In most models, 
the reset events have small probabilities, and this suggests that the QT-based 
method might be more suitable in this scenario.

\begin{figure}\label{fig:2}
	\begin{center}
		\begin{tabular}{cc}
		\resizebox{6cm}{!}{
		\begin{tikzpicture}
	\begin{loglogaxis}[
		legend pos = north east, width = .8\linewidth, height = .5\textheight,
		xlabel = {Time (sec)}, 
		ylabel = {Residual ~~ $\norm{A_{-1} + A_0 G + A_1G^2 - G}_\infty$}, title = {$\gamma = 0.95$}]
	    \addplot table {times_res_tandem_mam_rho_0.95_2.dat}; 
	    \addplot table {times_res_tandem_qt_rho_0.95_2.dat}; 
	    \legend{Truncation, QT}
	\end{loglogaxis}
	\end{tikzpicture}}
&	\resizebox{6cm}{!}{
	\begin{tikzpicture}
	\begin{loglogaxis}[
		legend pos = north east, width = .8\linewidth, height = .5\textheight,
		xlabel = {Time (sec)}, 
		ylabel = {Residual ~~ $\norm{A_{-1} + A_0 G + A_1G^2 - G}_\infty$}, title = {$\gamma = 0.99$}]
	    \addplot table {times_res_tandem_mam_rho_0.99.dat}; 
	    \addplot table {times_res_tandem_qt_rho_0.99.dat}; 
	    \legend{Truncation, QT}
	\end{loglogaxis}
	\end{tikzpicture}}
\end{tabular}	\caption{Two-node Jackson network with reset. Time versus accuracy of the QT-based method and of the truncation method relying on the {\tt SMCSolver} toolbox. For a small reset probability $1-\gamma$, the timings of the QT-based approach are much lower than the corresponding ones obtained upon truncation of the size combined with {\tt SMCSolver}.}
\end{center}
\end{figure}
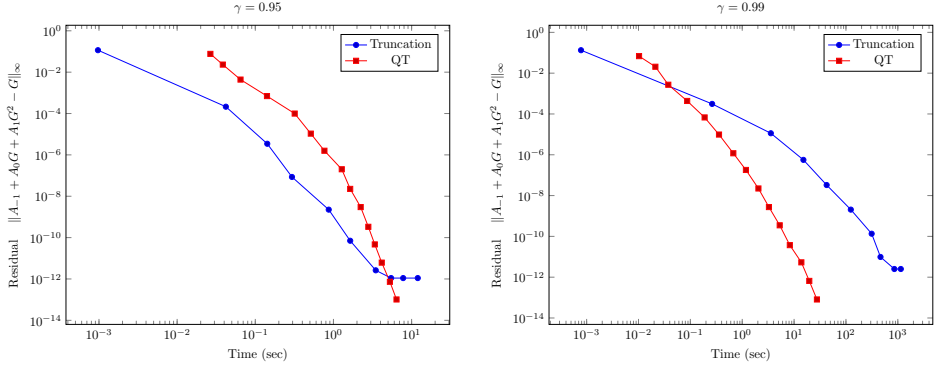

The case of finite but large queuing capacity networks can be treated with the same technique by relying on the QT-arithmetic for finite QT-matrices of the \cqttoolbox{} of \cite{bmr}.

\subsection{A Quasi-Birth-and-Death problem}
Consider a discrete-time Markov chain with state space $\mathbb N^2$
which models a random walk in the quarter plane, as described in Section~\ref{sec:qbd}. 
In \cite{zsh} a continuous
time model is analyzed, defined by the parameters
$a,b,\lambda>0, \theta=(a+b+\lambda)^{-1}$, which leads to the matrices
$A_{-1}=b\theta e_1e_1^T$,
$A_0=b\theta Z+a\theta Z^T$, $Z=\qtoep(0,0;1,0,0)$,
$A_1=\lambda\theta I$. In this case, the minimal nonnegative solution of
\eqref{eq:mateqG1} is $G=ee_1^T$, which belongs to $\EQT\setminus\QTD$.  

Here we treat a more general case, 
where the coefficients of \eqref{eq:mateqG1} belong to
$\QTD$ but $G$ does not and it is not explicitly known.
More specifically, we consider the two cases defined by:
\begin{align}
&A_{-1}=\frac19 \qtoep(3,3;2,0,1),\quad A_0=\frac19\qtoep(1,1;1, 0, 1),\quad A_1=\frac1{19}\qtoep(0,1;2,1,1), \label{eq:qbd1}\\
&A_{-1}=\frac1{16} \qtoep(5,5;2,0,1),\quad A_0=\frac1{16}\qtoep(2,2;7, 0, 2),\quad A_1=\frac1{16}\qtoep(1,1;2,1,1). \label{eq:qbd2}
\end{align}
Since $a_{-1}(1)<a_1(1)$, the minimal nonnegative solution $G$ of \eqref{eq:mateqG1} belongs to $\EQT\setminus\QTD$.  In particular, any
approximation $\hat G$ of $G$ in $\QTD$ will be affected by an
error $\norm{\hat G - G}_\infty \ge 1$. 

We have computed an approximation of the minimal nonnegative solution $G$ by applying the functional iteration $X_{k+1}=(I-A_0)^{-1}(A_{-1}+A_1X_k^2)$ analyzed in \cite{bmmj},  with starting approximation $X_0=T(g)+(I-T(g)) ee_1^T$.
In  Table \ref{tab:qbd} we report the features of the solution $G=T(g)+E_g+e v^T$ in the two cases. More specifically, we report the integers $n_-$ and $n_+$ such that $g_i<\epsilon$ for $i<-n_-$ or $i>n_+$, for $\epsilon=2^{-53}$ being the machine precision; the values $m,n$ such that $|c_{i,j}|<\epsilon$ for $i>m$ or for $j>n$, where $E_g=(c_{i,j})_{i,j \in \mathbb Z^+}$ and the rank $r$ of the $m\times n$ leading submatrix of $E_g$; the value $k$ such that $|v_i|<\epsilon$ for $i>k$. 
In Figure \ref{fig:mesh} we report a plot of the $200\times 200$ submatrix of the solution $G$ for the coefficients \eqref{eq:qbd1}. We may note the Toeplitz part $T(g)$ and the decay of the entries of the vector $v$.

\begin{table}\label{tab:qbd}
\begin{center}
\begin{tabular}{c|cccccc}
coefficients&$n_-$&$n_+$&$m$&$n$&$r$&$\ell$\\ \hline
\eqref{eq:qbd1}&617&46&859&52&9&29\\
\eqref{eq:qbd2}&1991&27&2874&31&12&52\\
\end{tabular}
\end{center}\caption{Numerical features of the symbol $g(z)=\sum_{i=-n_-}^{n_+}g_iz^i$, the correction $E_g\in\mathbb R^{m\times n}$ with rank $r$ and of the vector $v\in\mathbb R^\ell$ for  the solution $G=T(g)+E_g+ev^T$ for the QBD problems \eqref{eq:qbd1} and \eqref{eq:qbd2}. }
\end{table}

\begin{figure}\label{fig:mesh}
\begin{center}
\includegraphics[scale=0.28]{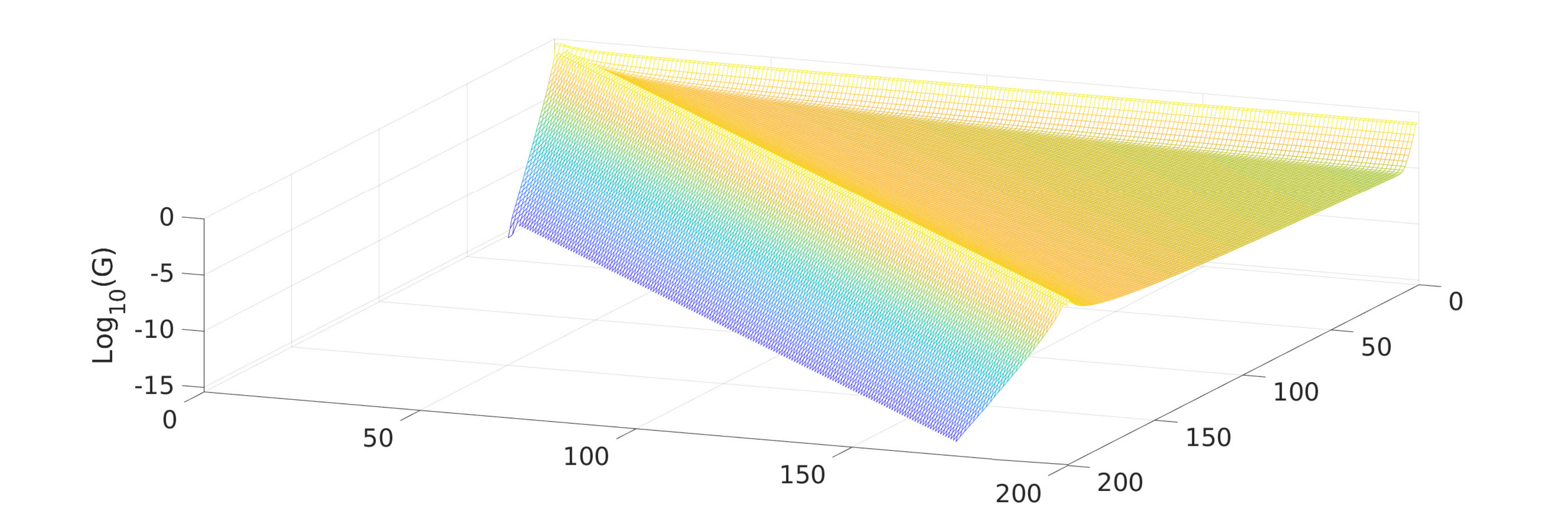}\caption{Log-plot of the $200\times 200$ submatrix of the solution $G$ for the coefficients \eqref{eq:qbd1}. }
\end{center}
\end{figure}

 We have compared our approach (QT-based) with the approximation obtained by truncating $A_{-1}$, $A_0$ and $A_1$ to finite size, as described in Section~\ref{sec:tandem}.
 In Figure \ref{fig:eqt} we plot the pairs (CPU time, residual errors) for the two different approaches. It is interesting to observe that the method based on truncation cannot reach a sufficiently accurate approximation. For the first problem, the CPU time required by the method based on truncation for reaching the best accuracy {\tt 9.0e-12} is about 122 seconds, while the time taken by our approach to reach the same precision is 5.22 seconds for a speed-up of 23.4. Moreover, our method reaches the best accuracy {\tt 1.3e-13} in 8.58 seconds.
For the second problem the differences are even more evident. The method based on truncation takes 945.3 seconds to reach the accuracy {\tt 2.3e-9} while our method takes 5.48 seconds to approximate the solution with the same accuracy. The speed-up in this case is 172.5. Moreover, our method reaches the highest precision of {\tt 1.0e-13} in 27.86 seconds. Also in this problem, 
all the residuals are measured as $\norm{A_{-1} + A_0 G + A_1G^2 - G}_\infty$.

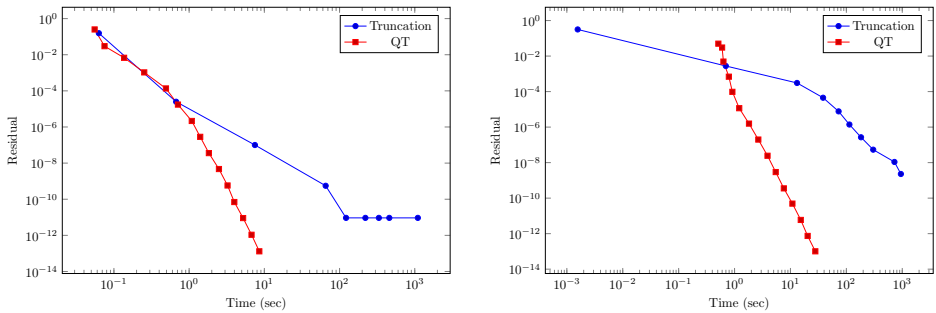
\begin{figure}\label{fig:eqt}
	\begin{center}
	\begin{tabular}{cc}
	\resizebox{6cm}{!}{
	\begin{tikzpicture}
	\begin{loglogaxis}[
		legend pos = north east, width = .8\linewidth, height = .45\textheight,
		xlabel = {Time (sec)}, 
		ylabel = {Residual}, title = {}]
	    \addplot table {example3_mam_2.dat}; 
	    \addplot table {example3_qt_2.dat}; 
	    \legend{Truncation, QT}
	\end{loglogaxis}
	\end{tikzpicture}}
	&
	\resizebox{6cm}{!}{
	\begin{tikzpicture}
	\begin{loglogaxis}[
		legend pos = north east, width = .8\linewidth, height = .45\textheight,
		xlabel = {Time (sec)}, 
		ylabel = {Residual}, title = {}]
	    \addplot table {example3_mam_3.dat}; 
	    \addplot table {example3_qt_3.dat}; 
	    \legend{Truncation, QT}
	\end{loglogaxis}
	\end{tikzpicture}}
	\end{tabular}
\end{center}\caption{A QBD example where $G\in\EQT$. 
Time versus accuracy of the QT-based method and of the truncation method relying on the {\tt SMCSolver} toolbox. 
On the left, the QBD defined by coefficients \eqref{eq:qbd1}; on the right, the case defined by \eqref{eq:qbd2}.}
\end{figure}

\section{Conclusions}\label{sec:conc}

We have introduced a computational framework for handling classes of
structured semi-infinite matrices encountered in the analysis of
random walks in the quarter plane which include rare events as reset
and catastrophes. This framework consists of two matrix classes $\QTD$
and $\EQT$ which extend the quasi Toeplitz matrices introduced in
\cite{bmm} and \cite{bmmr}. We proved that both classes are Banach
algebras, that matrices in these classes can be approximated to any
arbitrary precision in the infinite norm with a finite number of
parameters and that a finite arithmetic can be designed and
implemented by extending the \cqttoolbox{} of \cite{bmr}.  In
particular the computation of the invariant probability measure,
performed by means of the matrix analytic approach of
\cite{neuts:book} can be achieved by solving a quadratic matrix
equation with coefficients in the classes $\QTD$ or $\EQT$. We have
given conditions on the probabilities of the random walk under which
the minimal nonnegative solution $G$ of such quadratic matrix
equations belongs either to $\QTD$ or to $\EQT$. Examples of
algorithms for computing $G$ are given.  Numerical experiments, applied to significant problems, show
the effectiveness of our approach.

Some issues are still left to investigate. Namely, the analysis of the more general case where the coefficients $A_i=T(a_i)+E_i$ have a banded structure, that is $a_i(z)$ is a general Laurent polynomial; the study of the specific features of the solution $G$ when $a_{-1}(1)=a_1(1)$; and the challenging case of multidimensional random walks with more than two coordinates where the matrix coefficients $A_i$ have a multilevel structure.


\bibliographystyle{abbrv}

\end{document}